\documentclass[11pt]{amsart}
\usepackage{amsmath}
\usepackage{amssymb}
\usepackage{graphicx}
\usepackage{url}
\usepackage{multirow}

\usepackage{algorithm}
\usepackage{algorithmic}
\usepackage{float}
\usepackage{comment}
\usepackage[font=small,labelfont=bf]{caption}
\usepackage{subcaption}

\usepackage{color}
\usepackage{hyperref}
\usepackage{multicol}

\newtheorem{theorem}{Theorem}[section]

\newtheorem{lemma}[theorem]{Lemma}

\theoremstyle{definition}
\newtheorem{definition}[theorem]{Definition}

\newtheorem{remark}[theorem]{Remark}

\newtheorem{alg}[theorem]{Algorithm}

\numberwithin{equation}{section}

\title[Quasar Convex problems with functional constraints]{Mirror Descent Methods for Quasar Convex Optimization Problems With Non-Smooth Inequality
Constraints}

\author[M.~S.~Alkousa]{Mohammad S. Alkousa}
\address[M.~S.~Alkousa]{Innopolis University, Russia.}
\email{\tt m.alkousa@innopolis.ru}

\thanks{Preprint under updating.}


\begin{document}

\begin{abstract}
In this paper, we consider constraint optimization problems subject to non-smooth convex functional (inequality-type) constraints, wherein the objective function is non-smooth and quasar convex. We propose and analyze two groups of algorithms, each consisting of a standard version and a modified variant, that operate by switching between two types of iteration points: productive and non-productive. Within each group, we develop distinct mirror descent-type algorithms for both deterministic and stochastic settings, and we establish their convergence rates.
\end{abstract}

\maketitle

\section{Introduction}\label{sect:intro}

The optimization of non-smooth functions subject to constraints arises frequently in large-scale problems and their applications \cite{Nemirovski_Robust,Shpirko_primal}. Several methods exist for solving such problems, including mirror descent methods, which originated in \cite{Nemirovskii1979efficient,Nemirovsky1983Complexity}. The mirror descent method is widely regarded as a non-Euclidean generalization of subgradient methods, the latter being considered in \cite{shor1967generalized} for deterministic unconstrained problems in Euclidean settings, and for constrained problems in \cite{polyak1967general}. An extension of the mirror descent method to constrained problems was proposed in \cite{article:beck_comirror_2010,Nemirovsky1983Complexity}. Furthermore, the mirror descent method is also applicable to optimization problems in Banach spaces, where gradient descent is not directly applicable \cite{article:doan_2019}.

In \cite{Bayandina2018Mirror}, several optimal mirror descent algorithms were proposed for problems with functional constraints, featuring both adaptive step sizes and stopping rules. Subsequently, in \cite{stonyakin2018}, modifications of some of these algorithms were considered for the case of problems with many (large number) functional constraints. The literature on first-order methods for convex optimization problems with convex functional constraints includes works on the deterministic setting \cite{Bayandina2018Mirror,Fercoq2019Almost,stonyakin2018,stonyakin2019,Stonyakin2019Some,Savchuk2024some,Titov2019Mirror} as well as the stochastic setting \cite{Alkousa2020Modification,Alkousa2019some,Bayandina2017Stochastic,Zhou2020Algorithms,Xu2020Primal}.

A fundamental problem in contemporary machine learning and data science is the characterization of non-convex function classes that admit tractable optimization landscapes. Within this context, the class of quasar convex functions (see Definition \ref{defen_quasar}) has emerged as a prominent class, finding application in a broad spectrum of problems, ranging from linear dynamical system learning \cite{Hardt2018Gradient}, several generalized linear models \cite{Foster2018Uniform,Wang2023Continuized}, and phase retrieval \cite{Wang2019Robust} to empirical risk minimization \cite{Lee2016Optimizing} and general machine learning \cite{Farzin2025Minimisation,Wang2019Robust}. Despite a rich body of literature addressing the smooth regime (see, e.g., \cite{Guminov2023Accelerated,Hardt2018Gradient,Hermant2024Study,Hinder2020Near,Rubio2026smooth,Wang2019Robust}), the non-smooth counterpart remains comparatively underexplored (see, e.g., \cite{Ahookhosh2026Quasar,Brito2025Extending,Kabgani2026Robust}), particularly for problems involving functional constraints. This motivates the present study, which aims to advance the research in this direction.

The class of quasar convex functions was introduced in \cite{Hardt2018Gradient} under the name \textit{weak quasi-convexity}, wherein the authors provided an analysis of stochastic gradient descent for a weakly smooth quasar convex objective. Subsequently, \cite{Guminov2023Accelerated} proposed an accelerated algorithm for smooth quasar convex functions involving a search over a two-dimensional affine space. Later, \cite{Hinder2020Near} introduced an accelerated algorithm with line search, providing a bound on the runtime of a binary search procedure that required controlling the algorithm's iterates; the authors further established nearly matching lower bounds and introduced the term \textit{quasar convexity}. In a related development, \cite{Wang2023Continuized} extended the continuized acceleration framework originally developed in \cite{Even2021Continuized}. \cite{Lezane2024Mirror} generalized accelerated methods to non-Euclidean settings and function classes with weak smoothness. Finally, \cite{Jikai2020convergence} studied gradient norm minimization under the quasar-convexity assumption.

In this paper, we study non-convex optimization problems with convex functional constraints, where the objective function is non-smooth and quasar convex. We propose two groups of algorithms, each comprising distinct mirror descent-type methods in both deterministic and stochastic settings. In addition, we establish their convergence rates.

The paper consists of an introduction and four main sections. In Sect.~\ref{sect_basics}, we formulate the problem statement and recall the relevant fundamental facts and tools about mirror descent and the various components of the problems under consideration. Sect. \ref{section_group_I} is devoted to the first group of proposed algorithms (Algorithms \ref{alg:adaptive_mirror1} -- \ref{alg:stocadaptive_mirror1_mod}) along with their convergence analysis. In Sect. \ref{section_group_II}, analogously to the previous section, we analyze all proposed algorithms of the second group (Algorithms \ref{alg_determ_II} -- \ref{alg2_stoc_II} with their modifications). For each group, we propose a standard deterministic mirror descent algorithm together with its modification, as well as their stochastic counterparts.

\section{Fundamentals}\label{sect_basics}

Let $(\mathbf{E},\|\cdot\|)$ be a finite-dimensional normed vector space equipped with an arbitrary norm $\|\cdot\|$, and let $\mathbf{E}^*$ denote its conjugate space endowed with the dual norm
$$
    \|y\|_{*}=\max\limits_{x \in \mathbf{E}}\{\langle y, x\rangle: \|x\|\leq 1\},
$$
where $\langle y,x\rangle$ denotes the value of the continuous linear functional $y \in \mathbf{E}^*$ at $x \in \mathbf{E}$.

Let $Q \subset \mathbf{E}$ be a closed convex set, and let $d: Q \longrightarrow \mathbb{R}$ be a proper, closed, differentiable, and $1$-strongly convex function, referred to as a prox-function or distance-generating function. The associated Bregman divergence is defined as 
$$
    V_{x} (y) = V(y, x) := d(y) - d(x) - \langle \nabla d(x), y - x \rangle, \quad \forall x, y \in Q. 
$$
The Bregman divergence satisfies the following inequality
\begin{equation}\label{eq_breg}
    V(y, x) \geq \frac{1}{2} \|y - x\|^2, \quad \forall x, y \in Q. 
\end{equation}

In what follows, we denote by $\partial f(x)$ the subdifferential of a function $f$ at a point $x$, and by $\nabla f(x) \in \partial f(x)$ an arbitrary subgradient of $f$ at $x$.  

\begin{definition}[Quasar convex function \cite{Rubio2026smooth}]\label{defen_quasar}
For \(\gamma \in (0, 1]\), the function  \(f : Q \longrightarrow \mathbb{R}\) is called \(\gamma\)-quasar convex in $Q$  with respect to (or with center) \(x_* \in Q\), if $x_* \in \arg\min_{x \in Q} f(x)$, and
\begin{equation}\label{def_quasar}
    f(x_*) \geq f(x) + \frac{1}{\gamma} \langle \nabla f(x), x_* - x \rangle\quad \forall x \in Q.
\end{equation}

If $\gamma = 1$, the function is said to be star-convex \cite{Lee2016Optimizing,Nesterov2006Cubic}. Also, if $\gamma = 1$ and \eqref{def_quasar} holds for any $y \in Q$ rather than only for the center $x_*$, we conclude that $f$ is a convex function. 
\end{definition}

\textbf{Problem setting.} In this paper, we focus on a general class of problems, which is a class of optimization problems with functional (inequality-type) constraints 
\begin{equation}\label{general_problem_1}
    \min\left\{f(x): \;\; x \in Q \;\;\text{and} \;\; g_i(x)\leq 0 \;\; \text{for all} \;\; i = 1 , 2, \ldots, m \right\},
\end{equation}
where $f$ is a $\gamma$-quasar convex and $g_i$ (for all $i =1, \ldots, m$) are non-smooth convex functions given on $Q$. 

It is clear that instead of a set of functions  $\{g_i(\cdot)\}_{i=1}^{m}$ we can see one functional constraint $g: Q \longrightarrow \mathbb{R}$, such that $g(x) = \max_{1 \leq i \leq m}\{g_i(x)\}$. Therefore, by this setting, the problem \eqref{general_problem_1} will be equivalent to the following constrained minimization problem  
\begin{equation}\label{general_problem}
    \min\limits_{x\in Q,\, g(x) \leq 0} f(x).
\end{equation} 

We consider problem \eqref{general_problem} when the objective function \(f\) and the inequality constraint \(g\) are both non-smooth and Lipschitz-continuous, that is there exist $M_f > 0, M_g > 0$, such that
\begin{equation}\label{Lip_f}
    |f(x) - f(y)| \leq M_f \|x - y\| \quad \forall x, y \in Q,
\end{equation}
\begin{equation}\label{Lip_g}
    |g(x) - g(y)| \leq M_g \|x - y\| \quad \forall x, y \in Q.
\end{equation}
This means that $\|\nabla f(x)\|_* \leq M_f$ and $\|\nabla g(x)\|_* \leq M_g$, for all $x \in Q$. 

We also introduce an assumption: there exist a constant \(\Theta_0 > 0\), such that
\begin{equation}\label{bound_bregman}
    \sup_{x, y \in Q} V(y, x) \leq \Theta_0^2.
\end{equation}

We assume that the problem \eqref{general_problem} (and \eqref{general_problem_1}) is regular or feasible. This means that there exists a point \(x_0\) in the interior of \(Q\) (i.e., $x \in Q^{\circ}$) such that \(g(x_0) \leq 0\) (this means that $g_i(x_0) \leq \varepsilon\, \forall i = 1, \ldots, m$). 

Let $x_*$ be a solution of \eqref{general_problem}, we say that $\hat{x} \in Q$ is an $\varepsilon$-solution to \eqref{general_problem} if $f(\hat{x}) - f(x^*) \leq \varepsilon$ and $g(\hat{x}) \leq \varepsilon$. 

The proximal mapping operator is defined as follows
$$
    \operatorname{Mirr}_x(p) = \arg\min_{u \in Q} \left\{ \langle p, u \rangle + V(u, x) \right\} \quad \text{for each } x \in Q, \, p \in \mathbf{E}^*. 
$$
We assume, for simplicity, that $\operatorname{Mirr}_x(p)$ is easily computable.

\medskip 

In what follows, we present a lemma that plays a key role in the analysis of the proposed algorithms (in a deterministic setting) in the next sections. 

\begin{lemma}\label{main_lemma_determ}
Let \(f : Q \longrightarrow \mathbb{R}\), \(h > 0\), \(x \in Q^{\circ}\), and \(z = \operatorname{Mirr}_x(h\nabla f(x))\). Then for any \(y \in Q\), it holds
\begin{equation}\label{main_lemma_1}
    h\langle \nabla f(x), x - y \rangle \leq \frac{h^2}{2} \|\nabla f(x)\|_{*}^2 + V(y, x) - V(y, z).
\end{equation}
If \(y = x_*\) and \(f\) is \(\gamma\)-quasar convex, then it holds the following
\begin{align}
    h(f(x) - f(x_*)) & \leq \frac{h}{\gamma} \langle \nabla f(x), x - x_* \rangle  \label{eq_1lemma}
    \\& \leq \frac{1}{\gamma} \left( \frac{h^2}{2} \|\nabla f(x)\|_{*}^2 + V(x_*, x) - V(x_*, z) \right). \label{eq_2lemma}
\end{align}
\end{lemma}
\begin{proof}
The proof of \eqref{main_lemma_1} can be found, for instance, in \cite{Bayandina2018Mirror}. Inequality \eqref{eq_1lemma} follows from the quasar convexity of the function $f$, see \eqref{def_quasar}. Inequality \eqref{eq_2lemma} is a direct consequence of \eqref{main_lemma_1} with $y = x_*$, together with the fact that $\gamma \in (0, 1]$. 
\end{proof}

\section{Mirror Descent Algorithms for Quasar Convex Problems With Functional Constraints (Group I)}\label{section_group_I}

\subsection{Standard Adaptive Deterministic Mirror Descent Algorithm for Problems with One Functional Constraint.}

In this subsection, we propose an adaptive algorithm in the deterministic setting for solving problem \eqref{general_problem}. The algorithm is presented as Algorithm \ref{alg:adaptive_mirror1}. While this algorithm was originally proposed in \cite{Viscarra2026Convergence}, the convergence analysis presented here is entirely different and novel.

The output point of Algorithm \ref{alg:adaptive_mirror1} is selected from among the iterates $x_k$ satisfying $g(x_k) \leq \varepsilon$. Accordingly, a step $k$ is termed \textit{productive} if $g(x_k) \leq \varepsilon$; conversely, if the inequality is reversed, i.e., $g(x_k) > \varepsilon$, the step is termed \textit{non-productive}. Let $I$ and $J$ denote the sets of productive and non-productive steps, respectively, and let $|I|$ and $|J|$ denote their respective cardinalities.

\begin{algorithm}
\caption{Standard Adaptive Deterministic Mirror Descent Algorithm for Problems with One Functional Constraint.}
\label{alg:adaptive_mirror1}
\begin{algorithmic}[1]
\REQUIRE accuracy \(\varepsilon > 0\), starting point $x_0$, prox-function $d(\cdot)$, \(\Theta_0\) s.t. \eqref{bound_bregman} holds.
\STATE Initialize the set \(I\) as empty set.
\STATE Set \(k = 0\).
\REPEAT
\IF{\(g(x_k) \leq \varepsilon\)}
    \STATE \(M_k = \|\nabla f(x_k)\|_{*}\),
    \STATE \(h_k = \Theta_0 \left( \sum_{i=0}^k M_i^2 \right)^{-1/2}\),
    \STATE \(x_{k+1} = \operatorname{Mirr}_{x_k} \left( h_k \nabla f(x_k) \right)\), \quad “productive step”
    \STATE Add \(k\) to \(I\).
    \ELSE
    \STATE \(M_k = \|\nabla g(x_k)\|_{*}\),
    \STATE \(h_k = \Theta_0 \left( \sum_{i=0}^k M_i^2 \right)^{-1/2}\),
    \STATE \(x_{k+1} = \operatorname{Mirr}_{x_k} \left( h_k \nabla g(x_k) \right)\). \quad “non-productive step”
\ENDIF
\STATE Set \(k = k + 1\).
\UNTIL{\(k \geq \frac{2\Theta_0}{\varepsilon} \left(1 + \frac{1}{\gamma}\right) \left( \sum_{i=0}^{k-1} M_i^2 \right)^{1/2}\). }
\STATE \textbf{Output:} \(\hat{x} = \arg\min_{i \in I} f(x_i)\).
\end{algorithmic}
\end{algorithm}

For Algorithm \ref{alg:adaptive_mirror1}, we have the following result.

\begin{theorem}\label{theo_alg1_determ}
Let $f: Q \longrightarrow \mathbb{R}$ be a Lipschitz continuous \eqref{Lip_f}, $\gamma$-quasar convex function \eqref{def_quasar}, and $g: Q \longrightarrow \mathbb{R}$ be a convex Lipschitz continuous function \eqref{Lip_g}. Then, for Algorithm \ref{alg:adaptive_mirror1}, $I \ne \emptyset$, it stops after not more than
\[
    k = \left\lceil \frac{4 \max\{M_f^2, M_g^2\} \Theta_0^2 \left(1 + \frac{1}{\gamma}\right)^2}{\varepsilon^2} \right\rceil 
\]
iterations and returns \(\hat{x}\), such that
\begin{equation}\label{eq_results_the1}
    f(\hat{x}) - f(x_*) \leq  \varepsilon, \quad \text{and} \quad  g(\hat{x}) \leq \varepsilon.
\end{equation}
\end{theorem}

\begin{proof}
For each iteration \(i = 0, 1, \ldots, k-1\), define
\[
    \nabla_i := 
    \begin{cases}
    \nabla f(x_i) & \text{if } \, i \in I,\\
    \nabla g(x_i) & \text{if } \, i \in J.
    \end{cases}
\]

At each iteration, we have
\begin{equation}\label{M_i}
    h_k = \Theta_0 \left(\sum_{i = 0}^{k} M_i^2\right)^{-1/2}, \quad M_i = \begin{cases}
    \|\nabla f(x_i)\|_* & \text{if } \, i \in I,\\
    \|\nabla g(x_i)\|_* & \text{if } \, i \in J.
    \end{cases}
\end{equation}
Let \(V_i := V(x_*, x_i)\). Note that \(V_i \ge 0\) by \eqref{eq_breg}, and  \(V_i \le \Theta_0^2\) by \eqref{bound_bregman}.

By Lemma \ref{main_lemma_determ} (inequality \eqref{main_lemma_1}) with \(y = x_*\), at iteration \(i= 0, 1, \ldots, k-1\), and using \eqref{M_i}, we get
\begin{equation}\label{eq_main_all_iters}
     h_i \langle \nabla_i, x_i - x_* \rangle \leq \frac{h_i^2}{2} M_i^2 + V_i - V_{i+1}.
\end{equation}

For all $i \in I$, since $f$ is $\gamma$-quasar convex, we have
\begin{equation}\label{for_productive}
    \langle \nabla f(x_i), x_i - x_* \rangle \geq \gamma \bigl(f(x_i) - f(x_*)\bigr).
\end{equation}

For all \(i \in J\), by the convexity of \(g\) we have
\begin{equation*}
    \langle \nabla g(x_i), x_i - x_* \rangle \ge g(x_i) - g(x_*).
\end{equation*}

Thus, for all \(i = 0, 1, \ldots, k-1\), we have
\begin{equation}\label{hgh12}
    \langle \nabla_i, x_i - x_* \rangle \ge \delta_i,
\end{equation}
where
\[
    \delta_i := 
    \begin{cases}
    \gamma \bigl(f(x_i) - f(x_*)\bigr) & \text{if } \, i \in I,\\
    g(x_i) - g(x_*) & \text{if } \, i \in J.
    \end{cases}
\]

From \eqref{eq_main_all_iters} and \eqref{hgh12}, for all \(i = 0, 1, \ldots, k-1\), we obtain
\[
    h_i \delta_i \le \frac{h_i^2}{2} M_i^2 + V_i - V_{i+1}.
\]
Thus, 
\begin{equation}\label{jhjkh}
    \delta_i \le \frac{h_i}{2} M_i^2 + \frac{1}{h_i} \left(V_i - V_{i+1}\right) \quad \forall \, i = 0, 1, \ldots, k-1.
\end{equation}

Summing up \eqref{jhjkh} for \(i = 0, 1, \dots, k-1\), we get
\begin{equation}\label{lkl1}
    \sum_{i=0}^{k-1} \delta_i \le \sum_{i=0}^{k-1} \frac{h_i}{2} M_i^2 + \sum_{i=0}^{k-1} \frac{1}{h_i} \left(V_i - V_{i+1}\right).
\end{equation}

For the first sum in the right-hand side of \eqref{lkl1}, recall that for all $i = 0, 1, \dots, k-1$, we have \(h_i = \Theta_0 \left( \sum_{j=0}^i M_j^2 \right)^{-1/2}\). Therefore,
\[
    \frac{h_i}{2} M_i^2 = \frac{\Theta_0}{2} \cdot \frac{M_i^2}{\left( \sum_{j=0}^i M_j^2 \right)^{1/2}} \quad \forall \, i = 0,1, \ldots, k-1.
\]

Using the following standard inequality (which can be proved
by induction)
\[
    \sum_{i=0}^{k-1} \frac{M_i^2}{\left( \sum_{j=0}^i M_j^2 \right)^{1/2}} \le 2 \left( \sum_{i=0}^{k-1} M_i^2 \right)^{1/2},
\]
we get the following inequality 
\begin{equation}\label{bound1}
    \sum_{i=0}^{k-1} \frac{h_i}{2} M_i^2 \leq \Theta_0 \left( \sum_{i=0}^{k-1} M_i^2 \right)^{1/2}.
\end{equation}

For the second sum in the right-hand side of \eqref{lkl1}, we know 
\[
    \sum_{i=0}^{k-1} a_i (b_i - b_{i+1}) = a_0 b_0 + \sum_{i=1}^{k-1} (a_i - a_{i-1}) b_i - a_{k-1} b_k.
\]

Let \(a_i := 1/h_i\) and \(b_i := V_i \leq \Theta_0^2\) ($\forall i = 0,1, \ldots, k-1$). 
Then, we find
\begin{align*}
    \sum_{i=0}^{k-1} \frac{1}{h_i} (V_i - V_{i+1}) & \leq \frac{\Theta_0^2}{h_0} + \Theta_0^2 \sum_{i=1}^{k-1} \left(\frac{1}{h_i} - \frac{1}{h_{i-1}}\right) = \frac{\Theta_0^2}{h_{k-1}}.
\end{align*}

From the definition of \(h_i\), the sequence \(\{h_i\}_{i \geq 0}\) is non-increasing, because \(\sum_{j=0}^i M_j^2\) increases with \(i \geq 0\). Hence \(\{1/h_i\}_{i \geq 0}\) is non-decreasing sequence, so \(\frac{1}{h_i} - \frac{1}{h_{i-1}} \ge 0\) for all \(i \ge 1\). Therefore, 
\[
    \sum_{i=0}^{k-1} \frac{1}{h_i} (V_i - V_{i+1}) \le \frac{\Theta_0^2}{h_k}.
\]

Recall, \(\frac{1}{h_k} = \frac{1}{\Theta_0} \left( \sum_{i=0}^k M_i^2 \right)^{1/2}\). Since \(\left( \sum_{i=0}^k M_i^2 \right)^{1/2} \le \left( \sum_{i=0}^{k-1} M_i^2 \right)^{1/2} + M_k\), we obtain
\[
    \frac{\Theta_0^2}{h_k} \le \Theta_0 \left( \sum_{i=0}^{k-1} M_i^2 \right)^{1/2} + \Theta_0 M_k.
\]

Let \(M := \max\{M_f, M_g\}\), then \(M_k \le M\). Thus,
\begin{equation}\label{bound2}
    \sum_{i=0}^{k-1} \frac{1}{h_i} (V_i - V_{i+1}) \leq \Theta_0 \left( \sum_{i=0}^{k-1} M_i^2 \right)^{1/2} + \Theta_0 M.
\end{equation}

From \eqref{lkl1}, \eqref{bound1}, and \eqref{bound2} we get the following inequality
\begin{equation}\label{hjlk12}
    \sum_{i=0}^{k-1} \delta_i \le 2\Theta_0 \left( \sum_{i=0}^{k-1} M_i^2 \right)^{1/2} + \Theta_0 M.
\end{equation}

From the stopping criterion, we have
\[
    \left( \sum_{i=0}^{k-1} M_i^2 \right)^{1/2} \le \frac{\varepsilon k}{2\Theta_0 (1 + 1/\gamma)}.
\]

Thus, \eqref{hjlk12} will be as follows
\begin{equation}\label{dsfd0123}
    \sum_{i=0}^{k-1} \delta_i \leq \frac{\varepsilon k}{1 + 1/\gamma} + \Theta_0 M.
\end{equation}

For \(i \in J\), we have \(g(x_i) > \varepsilon\) and \(g(x_*) \le 0\), so
\[
    \delta_i = g(x_i) - g(x_*) \ge g(x_i) > \varepsilon.
\]

We have \(|I| + |J| = k\). So, we get the following
\begin{align}\label{545ss}
    \sum_{i=0}^{k-1} \delta_i &= \sum_{i \in I} \gamma \bigl(f(x_i)-f(x_*)\bigr) + \sum_{i \in J} (g(x_i)-g(x_*)) \nonumber
    \\& \ge \gamma \sum_{i \in I} (f(x_i)-f(x_*)) + |J| \varepsilon. 
\end{align}

Hence, from \eqref{dsfd0123} and \eqref{545ss}, we find 
\[
    \gamma \sum_{i \in I} (f(x_i)-f(x_*)) \leq \frac{\varepsilon (\gamma |I| - |J|)}{ 1 + \gamma } + \Theta_0 M.
\]

Since \(|J| \ge 0\), we have \(\gamma |I| - |J| \le \gamma |I|\). Thus, 
\[
    \sum_{i \in I} (f(x_i)-f(x_*)) \le \frac{\varepsilon |I|}{1 + \gamma } + \frac{\Theta_0 M}{\gamma}.
\]

Since \(\hat{x} = \arg\min_{i \in I} f(x_i)\), then \(f(\hat{x}) \le f(x_i)\) for all \(i \in I\) (in the next, we will prove that $I \ne \emptyset$). So, we get 
\[
    |I| \bigl(f(\hat{x}) - f(x_*)\bigr) \le \sum_{i \in I} (f(x_i)-f(x_*)) \le \frac{\varepsilon |I|}{1 + \gamma} + \frac{\Theta_0 M}{\gamma}.
\]
Thus,
\[
    f(\hat{x}) - f(x_*) \leq \frac{\varepsilon}{\gamma + 1} + \frac{\Theta_0 M}{\gamma |I|}.
\]

Our goal is to guarantee \( f(\hat{x}) - f(x_*) \leq \varepsilon \). This will hold if
\begin{equation}\label{tag33}
    \frac{\varepsilon}{\gamma + 1} + \frac{\Theta_0 M}{\gamma |I|} \leq \varepsilon \quad \implies \quad |I| \geq \frac{\Theta_0 M (\gamma + 1)}{\gamma^2 \varepsilon}.  
\end{equation}
So, if the number of productive steps \( |I| \) satisfies \eqref{tag33}, then  \( f(\hat{x}) - f(x_*) \leq \varepsilon \).

Since $f(x_i) - f(x_*) \geq 0$ for all $i = 0, 1, \ldots, k - 1$, from \eqref{dsfd0123}, and \eqref{545ss} we have
$$
    |J| \leq \frac{k \gamma}{1+ \gamma} + \frac{\Theta_0 M}{\varepsilon}.
$$
So, since $|I| + |J| = k$, we get
\begin{equation}\label{tag38}
    |I| \geq \frac{k}{\gamma + 1} - \frac{\Theta_0 M}{\varepsilon}.
\end{equation}

In the worst case for \( |I| \) (i.e., the smallest \( |I| \)), from the stopping criterion with \(M_i \le \max\{M_f, M_g\} =: M\) for all \(i = 0, 1, \ldots, k-1\), with equality (up to rounding), we have
\[
    k \approx \frac{2\Theta_0}{\varepsilon} \left( 1 + \frac{1}{\gamma} \right) M\sqrt{k} \quad \implies \quad k \approx  \frac{4\Theta_0^2 M^2(1+ \gamma)^2}{\varepsilon^2 \gamma^2}.
\]

By substituting this approximation in \eqref{tag38}, we find
\begin{equation}\label{tag41}
    |I| \geq  \frac{4\Theta_0^2 M^2 (1+ \gamma)}{\varepsilon^2 \gamma^2} - \frac{\Theta_0 M}{\varepsilon}.
\end{equation}

Let $A: = \frac{\Theta_0 M (\gamma + 1)}{\gamma^2 \varepsilon}.$ So, \eqref{tag33} is simply \( |I| \geq A \). From the right-hand side of \eqref{tag41}, we can write
\begin{align}\label{tag43}
    B & :=  \frac{4\Theta_0^2 M^2 (\gamma+1)}{\varepsilon^2 \gamma^2} - \frac{\Theta_0 M}{\varepsilon} = A \left( \frac{4\Theta_0 M}{\varepsilon} - \frac{\gamma^2}{\gamma+1} \right).
\end{align}
Since \( \gamma \in (0,1] \), we have \( \frac{\gamma^2}{\gamma+1} \leq \frac{1}{2} \). Thus, 
\[
    \frac{B}{A} = \frac{4\Theta_0 M}{\varepsilon} - \frac{\gamma^2}{\gamma+1} \geq \frac{4\Theta_0 M}{\varepsilon} - \frac{1}{2}.
\]

For sufficiently small \( \varepsilon > 0 \), the term \( \frac{4\Theta_0 M}{\varepsilon} \) dominates. So, $B \gg A$.

Hence, for sufficiently small \( \varepsilon > 0 \), the bound for $|I|$ in \eqref{tag41} (which is concluded from the stopping criterion of Algorithm \ref{alg:adaptive_mirror1}) is much stronger (i.e., larger) than the bound in \eqref{tag33} (which will we need to guarantee an $\varepsilon$-solution with respect to $f$). Consequently, any algorithm satisfying  \eqref{tag41} automatically satisfies \eqref{tag33}, and thus guarantees \( f(\hat{x}) - f(x_*) \leq \varepsilon \), which is the first desired inequality in \eqref{eq_results_the1}.

By the definition of the output $\hat{x}$, it is clear that $g(\hat{x}) \leq \varepsilon$, which is the second desired inequality in \eqref{eq_results_the1}.

Moreover, for sufficiently small \( \varepsilon \), from \eqref{tag41} (recall that this inequality is concluded from the stopping criterion of the algorithm), we find 
\begin{equation*}
    |I| \geq  \frac{4\Theta_0^2 M^2 (\gamma+1)}{\varepsilon^2 \gamma^2} - \frac{\Theta_0 M}{\varepsilon} > 0.
\end{equation*}
So, for sufficiently small \( \varepsilon \), we can conclude that \( |I| \geq 1 \) (i.e., there is at least one productive step).

To estimate the upper bound for the number of iterations, we have $M: = \max\{M_f, M_g\} \ge M_i$ for all \(i = 0, 1, \ldots. k-1\), and $\sum_{i=0}^{k-1} M_i^2 \le k M^2.$ The stopping criterion is satisfied when
\[
    k \ge \frac{2\Theta_0}{\varepsilon} \left(1 + \frac{1}{\gamma}\right) \sqrt{k M^2} = \frac{2\Theta_0 M}{\varepsilon} \left(1 + \frac{1}{\gamma}\right) \sqrt{k}.
\]
Thus, 
\[
    k \ge \frac{4\Theta_0^2 M^2}{\varepsilon^2} \left(1 + \frac{1}{\gamma}\right)^2.
\]

Therefore, the algorithm stops after not more than
\[
    k= \left\lceil \frac{4 \max\{M_f^2, M_g^2\} \Theta_0^2 \left(1 + \frac{1}{\gamma}\right)^2}{\varepsilon^2}  \right\rceil
\]
iterations. Hence, Theorem \ref{theo_alg1_determ} is proved. 
\end{proof}

\begin{remark}
We found that to achieve \( f(\hat{x}) - f(x_*) \leq \varepsilon \), the number of productive steps \( |I| \) must satisfy \eqref{tag33}, that is $|I| \geq \frac{\Theta_0 M (\gamma + 1)}{\gamma^2 \varepsilon} \geq \frac{\Theta_0 M (\gamma + 1)}{\varepsilon}.$ Therefore, by taking $\varepsilon \geq \frac{\Theta_0 M (\gamma + 1)}{|I|}$ as a stopping criterion for Algorithm \ref{alg:adaptive_mirror1} (which is non-adaptive), we get an $\varepsilon$-solution of the problem under consideration, and the algorithm will perform $\mathcal{O}(1/\varepsilon)$ productive steps. 
\end{remark}

\subsection{Modification of Standard Adaptive Deterministic Mirror Descent Algorithm for Problems with Many Functional Constraints.}\label{subsec_modif_groupI}

In this subsection, we consider problem \eqref{general_problem_1} with a large number of functional constraints. To handle this case, we modify Algorithm \ref{alg:adaptive_mirror1} in a manner analogous to that proposed in \cite{Stonyakin2018Adaptive}. Specifically, when a non-productive step occurs, i.e., when $g(x_k) > \varepsilon$, we do not compute a (sub)gradient of the max-type functional constraint $g(x) = \max_{1 \leq i \leq m} \{ g_i(x) \}$. Instead, we compute a (sub)gradient of a single functional constraint $g_i$ for which $g_i(x_k) > \varepsilon$. This modification is designed to reduce the algorithm's running time by avoiding the evaluation of all functional constraints on non-productive steps, which is particularly beneficial when the number of constraints $m$ is large. The resulting modified algorithm is presented as Algorithm \ref{alg:adaptive_mirror1_modific}.

\begin{algorithm}
\caption{Modification of Adaptive Deterministic Mirror Descent Algorithm for Problems with Many Functional Constraints.}
\label{alg:adaptive_mirror1_modific}
\begin{algorithmic}[1]
\REQUIRE accuracy \(\varepsilon > 0\), starting point $x_0$, prox-function $d(\cdot)$, \(\Theta_0\) s.t. \eqref{bound_bregman} holds.
\STATE Initialize the set \(I\) as empty set.
\STATE Set \(k = 0\).
\REPEAT
\IF{\(g_i(x_k) \leq \varepsilon  \, \forall i = 1, \ldots, m,  \)}
    \STATE \(M_k = \|\nabla f(x_k)\|_{*}\),
    \STATE \(h_k = \Theta_0 \left( \sum_{i=0}^k M_i^2 \right)^{-1/2}\),
    \STATE \(x_{k+1} = \operatorname{Mirr}_{x_k} \left( h_k \nabla f(x_k) \right)\), \quad “productive step”
    \STATE Add \(k\) to \(I\).
    \ELSE 
    \STATE (i.e., $ \exists q = q(k) \in \{1, \ldots, m\}, \text{s.t. } g_{q(k)}(x_k) > \varepsilon$)
    \STATE \(M_k = \|\nabla g_{q(k)}(x_k)\|_{*}\),
    \STATE \(h_k = \Theta_0 \left( \sum_{i=0}^k M_i^2 \right)^{-1/2}\),
    \STATE \(x_{k+1} = \operatorname{Mirr}_{x_k} \left( h_k \nabla g_{q(k)}(x_k) \right)\). \quad “non-productive step”
\ENDIF
\STATE Set \(k = k + 1\).
\UNTIL{\(k \geq \frac{2\Theta_0}{\varepsilon} \left(1 + \frac{1}{\gamma}\right) \left( \sum_{i=0}^{k-1} M_i^2 \right)^{1/2}\).}
\STATE \textbf{Output:} \(\hat{x} = \arg\min_{i \in I} f(x_i)\).
\end{algorithmic}
\end{algorithm}

For Algorithm \ref{alg:adaptive_mirror1_modific}, we have the following result.

\begin{theorem}\label{theo_modifalg1_determ}
Let $f: Q \longrightarrow \mathbb{R}$ be a Lipschitz continuous \eqref{Lip_f}, $\gamma$-quasar convex function \eqref{def_quasar}, and $g_i: Q \longrightarrow \mathbb{R}$ be  convex Lipschitz continuous functions with constants $M_{g_i} > 0$ \eqref{Lip_g}. Assume that for all $x \in Q$, we have $\|\nabla f(x)\|_* \leq M_f$, and $\|\nabla g_i(x)\|_* \leq M_g \, \forall i =1, \ldots, m$, where $M_g = \max_{1 \leq i \leq m} \{M_{g_i}\}$. Then, for Algorithm \ref{alg:adaptive_mirror1}, $I \ne \emptyset$, it stops after not more than
\[
    k = \left\lceil \frac{4 \max\{M_f^2, M_g^2\} \Theta_0^2 \left(1 + \frac{1}{\gamma}\right)^2}{\varepsilon^2} \right\rceil 
\]
iterations and returns \(\hat{x}\), such that
\begin{equation}\label{eps_sol_modifi_alg1}
    f(\hat{x}) - f(x_*) \leq  \varepsilon, \quad \text{and} \quad  g(\hat{x}) \leq \varepsilon.
\end{equation}
\end{theorem}

\begin{proof}
For each iteration \(i = 0, 1, \ldots, k-1\), define
\[
    \nabla_i := 
    \begin{cases}
    \nabla f(x_i) & \text{if } \, i \in I,\\
    \nabla g_{q(i)}(x_i) & \text{if } \, i \in J.
    \end{cases}
\]

At each iteration, we have
\begin{equation}\label{M_i_mod}
    h_k = \Theta_0 \left(\sum_{i = 0}^{k} M_i^2\right)^{-1/2}, \quad M_i = \begin{cases}
    \|\nabla f(x_i)\|_* & \text{if } \, i \in I,\\
    \|\nabla g_{q(i)}(x_i)\|_* & \text{if } \, i \in J.
    \end{cases}
\end{equation}
Let \(V_i := V(x_*, x_i)\). By Lemma \ref{main_lemma_determ} (inequality \eqref{main_lemma_1}) with \(y = x_*\) at iteration \(i= 0, 1, \ldots, k-1\), and using \eqref{M_i_mod}, we get
\begin{equation}\label{eq_main_all_iters_mod}
     h_i \langle \nabla_i, x_i - x_* \rangle \leq \frac{h_i^2}{2} M_i^2 + V_i - V_{i+1}.
\end{equation}

For all $i \in I$, since $f$ is $\gamma$-quasar convex, we have
\begin{equation*}
    \langle \nabla f(x_i), x_i - x_* \rangle \geq \gamma \bigl(f(x_i) - f(x_*)\bigr).
\end{equation*}

For all \(i \in J\), by the convexity of \(g_{q(i)}\), we have
\begin{equation*}
    \langle \nabla g_{q(i)}(x_i), x_i - x_* \rangle \geq g_{q(i)}(x_i) - g_{q(i)}(x_*).
\end{equation*}

Thus, for all \(i = 0, 1, \ldots, k-1\), we have
\begin{equation}\label{hgh12_mod}
    \langle \nabla_i, x_i - x_* \rangle \ge \delta_i,
\end{equation}
where
\[
    \delta_i := 
    \begin{cases}
    \gamma \bigl(f(x_i) - f(x_*)\bigr) & \text{if } \, i \in I,\\
    g_{q(i)}(x_i) - g_{q(i)}(x_*) & \text{if } \, i \in J.
    \end{cases}
\]

From \eqref{eq_main_all_iters_mod} and \eqref{hgh12_mod}, for all \(i = 0, 1, \ldots, k-1\), we obtain
\begin{equation}\label{jhjkh_mod}
    \delta_i \le \frac{h_i}{2} M_i^2 + \frac{1}{h_i} \left(V_i - V_{i+1}\right) \quad \forall \, i = 0, 1, \ldots, k-1.
\end{equation}

Summing up \eqref{jhjkh_mod} for \(i = 0, 1, \dots, k-1\), we get
\begin{equation}\label{lkl1_mod}
    \sum_{i=0}^{k-1} \delta_i \le \sum_{i=0}^{k-1} \frac{h_i}{2} M_i^2 + \sum_{i=0}^{k-1} \frac{1}{h_i} \left(V_i - V_{i+1}\right).
\end{equation}

For the first sum in the right-hand side of \eqref{lkl1_mod}, we have (see \eqref{bound1}) 
\begin{equation}\label{bound1_mod}
    \sum_{i=0}^{k-1} \frac{h_i}{2} M_i^2 \leq \Theta_0 \left( \sum_{i=0}^{k-1} M_i^2 \right)^{1/2}.
\end{equation}

For the second sum in the right-hand side of \eqref{lkl1_mod}, we have (see \eqref{bound2}) 
\begin{equation}\label{bound2_mod}
    \sum_{i=0}^{k-1} \frac{1}{h_i} (V_i - V_{i+1}) \leq \Theta_0 \left( \sum_{i=0}^{k-1} M_i^2 \right)^{1/2} + \Theta_0 M,
\end{equation}
where \(M := \max\{M_f, M_g\}\).

From \eqref{lkl1_mod}, \eqref{bound1_mod}, \eqref{bound2_mod}, and from stopping criterion, we get the following inequality
\begin{equation}\label{dsfd0123_mod}
    \sum_{i=0}^{k-1} \delta_i \leq \frac{\varepsilon k}{1 + 1/\gamma} + \Theta_0 M.
\end{equation}

Recall that for \(i \in J\), we have \(g_{q(i)}(x_i) > \varepsilon\) and \(g_{q(i)}(x_*) \le 0\), so
\[
    \delta_i = g_{q(i)}(x_i) - g_{q(i)}(x_*) \ge g_{q(i)}(x_i) > \varepsilon.
\]

We have \(|I| + |J| = k\). So, we get the following
\begin{align}\label{545ss_mod}
    \sum_{i=0}^{k-1} \delta_i & = \sum_{i \in I} \gamma \bigl(f(x_i)-f(x_*)\bigr) + \sum_{i \in J} (g_{q(i)}(x_i)-g_{q(i)}(x_*)) \nonumber
    \\& \ge \gamma \sum_{i \in I} (f(x_i)-f(x_*)) + |J| \varepsilon. 
\end{align}

Hence, from \eqref{dsfd0123_mod} and \eqref{545ss_mod}, we find 
\[
    \gamma \sum_{i \in I} (f(x_i)-f(x_*)) \leq \frac{\varepsilon (\gamma |I| - |J|)}{ 1 + \gamma } + \Theta_0 M.
\]

Thus, 
\[
    \sum_{i \in I} (f(x_i)-f(x_*)) \le \frac{\varepsilon |I|}{1 + \gamma } + \frac{\Theta_0 M}{\gamma}.
\]

From the definition of \(\hat{x} = \arg\min_{i \in I} f(x_i)\), we find
\[
    |I| \bigl(f(\hat{x}) - f(x_*)\bigr) \le \sum_{i \in I} (f(x_i)-f(x_*)) \le \frac{\varepsilon |I|}{1 + \gamma} + \frac{\Theta_0 M}{\gamma}.
\]
Therefore, 
\[
    f(\hat{x}) - f(x_*) \leq \frac{\varepsilon}{\gamma + 1} + \frac{\Theta_0 M}{\gamma |I|}.
\]

Now, the rest of the proof, to conclude \eqref{eps_sol_modifi_alg1}, is the same as the proof of Theorem \ref{theo_alg1_determ}.
\end{proof}

\subsection{Standard Adaptive Stochastic Mirror Descent Algorithm for Problems with One Functional Constraint.}\label{subsec_stoc_stand_GI}

In this subsection, we address the stochastic formulation of problem \eqref{general_problem} and propose a stochastic version of Algorithm \ref{alg:adaptive_mirror1}, presented as Algorithm \ref{alg:stocadaptive_mirror1}. Although this algorithm was originally introduced in \cite{Viscarra2026Convergence}, the convergence analysis provided herein is entirely different and novel.

For the stochastic setting of \eqref{general_problem}, we introduce the following standard assumptions (see \cite{Bayandina2018Mirror,Bayandina2017Stochastic}). For any given point $x \in Q$, we assume access to stochastic (sub)gradients $\nabla f(x,\xi)$ and $\nabla g(x,\zeta)$, where $\xi$ and $\zeta$ are random vectors. These stochastic (sub)gradients satisfy

\begin{equation}\label{bound_stoc_f_g_1}
    \mathbb{E}[\nabla f(x,\xi)] = \nabla f(x) \in \partial f(x), \quad \text{and} \quad \mathbb{E}[\nabla g(x,\zeta)] = \nabla g(x) \in \partial g(x),
\end{equation}
where $\mathbb{E}$ denotes the expectation, and
\begin{equation}\label{bound_stoc_f_g_2} 
    \|\nabla f(x,\xi)\|_* \leq M_f, \quad \text{and} \quad  \|\nabla g(x,\zeta)\|_* \leq M_g, \quad a.s. \;\, \text{in} \; \xi, \zeta.
\end{equation}

Let $\varepsilon > 0$ be a given positive real number. A (random) point $\hat{x} \in Q$ is said to be an expected $\varepsilon$-solution to problem \eqref{general_problem} in the stochastic setting if
\begin{equation}\label{expected_sol}
    \mathbb{E}[f(\hat{x})] - f(x^*) \leq \varepsilon, \quad \text{and} \quad g(\hat{x}) \leq \varepsilon. 
\end{equation}

We next present a stochastic version of Lemma \ref{main_lemma_determ}, which plays a key role in the analysis of the proposed algorithms (in the stochastic setting).

\begin{lemma}\label{main_lemma_stoch}
Let \(f : Q \longrightarrow \mathbb{R}\), \(h > 0\), \(x \in Q^{\circ}\), and \(z = \operatorname{Mirr}_x(h(\nabla f(x) + \Delta))\). Then, for any \(y \in Q\), it holds
\begin{equation}\label{main_lemma_stoc}
    h\langle \nabla f(x) + \Delta, x - y \rangle \leq \frac{h^2}{2} \|\nabla f(x) + \Delta\|_{*}^2 + V(y, x) - V(y, z).
\end{equation}

If \(y = x_*\), and \(f\) is \(\gamma\)-quasar convex, then it holds the following
\begin{align}
    & \quad \;\, h(f(x) - f(x_*)) + \frac{h}{\gamma} \langle \Delta, x - x_* \rangle  \nonumber
    \\& \leq \frac{h}{\gamma} \langle \nabla f(x) + \Delta, x - x_* \rangle   \label{eq_1lemma_stoch}
    \\& \leq \frac{1}{\gamma} \left( \frac{h^2}{2} \|\nabla f(x) + \Delta\|_{*}^2 + V(x_*, x) - V(x_*, z) \right). \label{eq_2lemma_stoch}
\end{align}
\end{lemma}
\begin{proof}
The proof of \eqref{main_lemma_stoc} can be found, for instance, in \cite{Bayandina2018Mirror}. Inequality \eqref{eq_1lemma_stoch} follows from the quasar convexity of the function $f$, see \eqref{def_quasar}, and \eqref{main_lemma_stoc}. Inequality \eqref{eq_2lemma_stoch} is a direct consequence of \eqref{main_lemma_stoc} with $y = x_*$, together with the fact that $\gamma \in (0, 1]$. 
\end{proof}

\begin{algorithm}
\caption{Standard Adaptive Stochastic Mirror Descent Algorithm for Problems with One Functional Constraint.}
\label{alg:stocadaptive_mirror1}
\begin{algorithmic}[1]
\REQUIRE accuracy \(\varepsilon > 0\), starting point $x_0$, prox-function $d(\cdot)$, \(\Theta_0\) s.t. \eqref{bound_bregman} holds.
\STATE Initialize the set \(I\) as empty set.
\STATE Set \(k = 0\).
\REPEAT
\IF{\(g(x_k) \leq \varepsilon\)}
    \STATE \(M_k = \|\nabla f(x_k, \xi_k)\|_{*}\),
    \STATE \(h_k = \Theta_0 \left( \sum_{i=0}^k M_i^2 \right)^{-1/2}\),
    \STATE \(x_{k+1} = \operatorname{Mirr}_{x_k} \left( h_k \nabla f(x_k, \xi_k) \right)\), \quad “productive step”
    \STATE Add \(k\) to \(I\).
    \ELSE
    \STATE \(M_k = \|\nabla g(x_k, \zeta_k)\|_{*}\),
    \STATE \(h_k = \Theta_0 \left( \sum_{i=0}^k M_i^2 \right)^{-1/2}\),
    \STATE \(x_{k+1} = \operatorname{Mirr}_{x_k} \left( h_k \nabla g(x_k, \zeta_k) \right)\). \quad “non-productive step”
\ENDIF
\STATE Set \(k = k + 1\).
\UNTIL{\(k \geq \frac{2\Theta_0}{\varepsilon} \left(1 + \frac{1}{\gamma}\right) \left( \sum_{i=0}^{k-1} M_i^2 \right)^{1/2}\).}
\STATE \textbf{Output:} \(\hat{x} = \arg\min_{i \in I} f(x_i)\).
\end{algorithmic}
\end{algorithm}

For Algorithm \ref{alg:stocadaptive_mirror1}, we have the following result.
 
\begin{theorem}\label{theorem_stoc_adaptive}
Let $f: Q \longrightarrow \mathbb{R}$ be a $\gamma$-quasar convex function \eqref{def_quasar}, and $g: Q \longrightarrow \mathbb{R}$ be a convex Lipschitz continuous function. Assume that \eqref{bound_stoc_f_g_1}, and \eqref{bound_stoc_f_g_2} hold. Then, for Algorithm \ref{alg:stocadaptive_mirror1}, $I \ne \emptyset$, it stops after not more than
\[
    k = \left\lceil \frac{4 \max\{M_f^2, M_g^2\} \Theta_0^2 \left(1 + \frac{1}{\gamma}\right)^2}{\varepsilon^2} \right\rceil 
\]
iterations and returns \(\hat{x}\), such that
\begin{equation}\label{eq_results_theo2}
    \mathbb{E}[f(\hat{x})] - f(x_*) \leq  \varepsilon, \quad \text{and} \quad g(\hat{x}) \leq \varepsilon.
\end{equation}
\end{theorem}
\begin{proof}

For each iteration \(i = 0, 1, \ldots, k-1\), define
\begin{equation}\label{nabla_i_Delta_i_stoc}
    \nabla_i := 
    \begin{cases}
    \nabla f(x_i, \xi_i) & \text{if } \, i \in I,\\
    \nabla g(x_i, \zeta_i) & \text{if } \, i \in J.
    \end{cases} \quad \Delta_i := \begin{cases}
    \nabla f(x_i, \xi_i) -\nabla f(x_i)  & \text{if } \, i \in I,\\
    \nabla g(x_i, \zeta_i) - \nabla g(x_i) & \text{if } \, i \in J.
    \end{cases} 
\end{equation}

At each iteration, we have
\begin{equation}\label{M_i_sto}
    h_k = \Theta_0 \left(\sum_{i = 0}^{k} M_i^2\right)^{-1/2}, \quad M_i = \begin{cases}
    \|\nabla f(x_i, \xi_i)\|_* & \text{if } \, i \in I,\\
    \|\nabla g(x_i, \zeta_i)\|_* & \text{if } \, i \in J.
    \end{cases}
\end{equation}
Let \(V_i := V(x_*, x_i)\). Note that \(V_i \ge 0\) by \eqref{eq_breg}, and  \(V_i \le \Theta_0^2\) by \eqref{bound_bregman}.

By Lemma \ref{main_lemma_stoch} (inequality \eqref{main_lemma_stoc}) with \(y = x_*\) at iteration \(i= 0, 1, \ldots, k-1\), and using \eqref{nabla_i_Delta_i_stoc}, \eqref{M_i_sto},  we get
\begin{equation}\label{eq_main_all_iters_stoc}
     h_i \langle \nabla_i, x_i - x_* \rangle \leq \frac{h_i^2}{2} M_i^2 + V_i - V_{i+1}.
\end{equation}

For all $i \in I$, since $f$ is $\gamma$-quasar convex, from \eqref{eq_1lemma_stoch} with $\Delta_i = \nabla f(x_i, \xi_i) - \nabla f(x_i)$, we have
\begin{equation}\label{for_productive_stoc}
    \langle \nabla f(x_i, \xi_i), x_i - x_* \rangle + \langle \nabla f(x_i, \xi_i) - \nabla f(x_i),  x_* - x_i \rangle \geq \gamma \left(f(x_i) - f(x_*)\right).
\end{equation}

For all \(i \in J\), by the convexity of \(g\), we have
\begin{equation*}
    \langle \nabla g(x_i), x_i - x_* \rangle \ge g(x_i) - g(x_*).
\end{equation*}
Thus, 
\begin{equation}\label{for_nonproductive_stoc1}
    \langle \nabla g(x_i, \zeta_i), x_i - x_* \rangle + \langle \nabla g(x_i, \zeta_i) - \nabla g(x_i),  x_* - x_i \rangle \geq g(x_i) - g(x_*).
\end{equation}

Therefore, for all \(i = 0, 1, \ldots, k-1\), we have
\begin{equation}\label{hgh12_s}
    \langle \nabla_i, x_i - x_* \rangle + \beta_i \geq \delta_i ,
\end{equation}
where
\[
    \delta_i :=  \begin{cases}
    \gamma \bigl(f(x_i) - f(x_*)\bigr) & \text{if } \, i \in I,
    \\ g(x_i) - g(x_*) & \text{if } \, i \in J,
    \end{cases}
\]
and 
\[
    \beta_i := \begin{cases}
     \langle \nabla f(x_i, \xi_i) - \nabla f(x_i),  x_* - x_i \rangle & \text{if } \, i \in I,
     \\ \langle \nabla g(x_i, \zeta_i) - \nabla g(x_i),   x_* - x_i \rangle & \text{if } \, i \in J.
    \end{cases}
\]

From \eqref{eq_main_all_iters_stoc} and \eqref{hgh12_s}, for all \(i = 0, 1, \ldots, k-1\), we obtain
\[
    h_i \delta_i \le \frac{h_i^2}{2} M_i^2 + V_i - V_{i+1} + h_i \beta_i. 
\]
Thus, 
\begin{equation}\label{jhjkh_s}
    \delta_i \le \frac{h_i}{2} M_i^2 + \frac{1}{h_i} \left(V_i - V_{i+1}\right) + \beta_i \quad \forall \, i = 0, 1, \ldots, k-1.
\end{equation}

Summing up \eqref{jhjkh_s} for \(i = 0, 1, \dots, k-1\), we get
\begin{equation}\label{lkl1_s}
    \sum_{i=0}^{k-1} \delta_i \le \sum_{i=0}^{k-1} \frac{h_i}{2} M_i^2 + \sum_{i=0}^{k-1} \frac{1}{h_i} \left(V_i - V_{i+1}\right) + \sum_{i=0}^{k-1} \beta_i.
\end{equation}

For the first sum in the right-hand side of \eqref{lkl1_s}, as in the proof of Theorem \ref{theo_alg1_determ}, we have (see \eqref{bound1})
\begin{equation}\label{bound1_s}
    \sum_{i=0}^{k-1} \frac{h_i}{2} M_i^2 \leq \Theta_0 \left( \sum_{i=0}^{k-1} M_i^2 \right)^{1/2}.
\end{equation}

Also, for the second sum in the right-hand side of \eqref{lkl1_s}, we have (see \eqref{bound2})
\begin{equation}\label{bound2_s}
    \sum_{i=0}^{k-1} \frac{1}{h_i} (V_i - V_{i+1}) \leq \Theta_0 \left( \sum_{i=0}^{k-1} M_i^2 \right)^{1/2} + \Theta_0 M,
\end{equation}
where \(M := \max\{M_f, M_g\}\). 

From \eqref{lkl1_s}, \eqref{bound1_s}, and \eqref{bound2_s} we get the following inequality
\begin{equation}\label{hjlk12_s}
    \sum_{i=0}^{k-1} \delta_i \le 2\Theta_0 \left( \sum_{i=0}^{k-1} M_i^2 \right)^{1/2} + \Theta_0 M + \sum_{i=0}^{k-1} \beta_i.
\end{equation}

From the stopping criterion, we have
\[
    \left( \sum_{i=0}^{k-1} M_i^2 \right)^{1/2} \le \frac{\varepsilon k}{2\Theta_0 (1 + 1/\gamma)}.
\]

Thus, \eqref{hjlk12_s} will be as follows
\begin{equation}\label{dsfd0123_s}
    \sum_{i=0}^{k-1} \delta_i \leq \frac{\varepsilon k}{1 + 1/\gamma} + \Theta_0 M + \sum_{i=0}^{k-1} \beta_i.
\end{equation}

We have \(|I| + |J| = k\). So (see \eqref{545ss}), 
\begin{align}\label{545ss_s}
    \sum_{i=0}^{k-1} \delta_i \geq \gamma \sum_{i \in I} (f(x_i)-f(x_*)) + |J| \varepsilon. 
\end{align}

Hence, from \eqref{dsfd0123_s} and \eqref{545ss_s}, we find 
\[
    \gamma \sum_{i \in I} (f(x_i)-f(x_*)) \leq \frac{\varepsilon (\gamma |I| - |J|)}{ 1 + \gamma } + \Theta_0 M + \sum_{i=0}^{k-1} \beta_i.
\]

Since $|J| \geq 0$, WE GET 
\[
    \sum_{i \in I} (f(x_i)-f(x_*)) \le \frac{\varepsilon |I|}{1 + \gamma } + \frac{\Theta_0 M}{\gamma} + \frac{1}{\gamma} \sum_{i=0}^{k-1} \beta_i.
\]

Consequently, 
\[
    f(\hat{x}) - f(x_*) \leq \frac{\varepsilon}{\gamma + 1} + \frac{\Theta_0 M}{\gamma |I|} + \frac{1}{\gamma} \sum_{i=0}^{k-1} \frac{\beta_i}{|I|}.
\]

Now, by taking the expectation, and since $\sum_{i=0}^{k-1} \mathbb{E} \left[\frac{\beta_i}{|I|}\right] = 0 $ (see \cite{Bayandina2017Stochastic}), we obtain
\[
    \mathbb{E}[f(\hat{x})] - f(x_*) \leq \frac{\varepsilon}{\gamma + 1} + \frac{\Theta_0 M}{\gamma |I|}.
\]

But, for sufficiently small $\varepsilon > 0$, in a similar way that was in the proof of Theorem \ref{theo_alg1_determ}, we can prove the desired result,
$$
    \mathbb{E}[f(\hat{x})] - f(x_*) \leq  \varepsilon.
$$
In addition $g(\hat{x}) \leq \varepsilon$, and $I \ne \emptyset.$
\end{proof}

\subsection{Modification of Standard Adaptive Stochastic Mirror Descent Algorithm for Problems with Many Functional Constraints.}\label{modif_stoc_GI}

In this subsection, we address the stochastic formulation of problem \eqref{general_problem_1} in the presence of a large number of functional constraints. In this setting, we modify Algorithm \ref{alg:stocadaptive_mirror1} in a manner similar to that proposed in \cite{Alkousa2020Modification}. The proposed modification is presented as Algorithm \ref{alg:stocadaptive_mirror1_mod}.

\begin{algorithm}
\caption{Modification of Adaptive Stochastic Mirror Descent Algorithm for Problems with One Functional Constraint.}
\label{alg:stocadaptive_mirror1_mod}
\begin{algorithmic}[1]
\REQUIRE accuracy \(\varepsilon > 0\), starting point $x_0$, prox-function $d(\cdot)$, \(\Theta_0\) s.t. \eqref{bound_bregman} holds.
\STATE Initialize the set \(I\) as empty set.
\STATE Set \(k = 0\).
\REPEAT
\IF{\(g_i(x_k) \leq \varepsilon \, \forall i = 1, \ldots, m, \)}
    \STATE \(M_k = \|\nabla f(x_k, \xi_k)\|_{*}\),
    \STATE \(h_k = \Theta_0 \left( \sum_{i=0}^k M_i^2 \right)^{-1/2}\),
    \STATE \(x_{k+1} = \operatorname{Mirr}_{x_k} \left( h_k \nabla f(x_k, \xi_k) \right)\), \quad “productive step”
    \STATE Add \(k\) to \(I\).
    \ELSE
    \STATE (i.e., $ \exists q = q(k) \in \{1, \ldots, m\}, \text{s.t. } g_{q(k)}(x_k) > \varepsilon$)
    \STATE \(M_k = \|\nabla g_{q(k)}(x_k, \zeta_k)\|_{*}\),
    \STATE \(h_k = \Theta_0 \left( \sum_{i=0}^k M_i^2 \right)^{-1/2}\),
    \STATE \(x_{k+1} = \operatorname{Mirr}_{x_k} \left( h_k \nabla g_{q(k)}(x_k, \zeta_k) \right)\). \quad “non-productive step”
\ENDIF
\STATE Set \(k = k + 1\).
\UNTIL{\(k \geq \frac{2\Theta_0}{\varepsilon} \left(1 + \frac{1}{\gamma}\right) \left( \sum_{i=0}^{k-1} M_i^2 \right)^{1/2}\).}
\STATE \textbf{Output:} \(\hat{x} = \arg\min_{ i \in I} f(x_i)\).
\end{algorithmic}
\end{algorithm}

For Algorithm \ref{alg:stocadaptive_mirror1_mod}, we have the following result.

\begin{theorem}\label{theorem_stoc_adaptive_mod}
Let $f: Q \longrightarrow \mathbb{R}$ be a Lipschitz continuous \eqref{Lip_f}, $\gamma$-quasar convex function \eqref{def_quasar}, and $g_i: Q \longrightarrow \mathbb{R}$ be  convex Lipschitz continuous functions with constants $M_{g_i} > 0$ \eqref{Lip_g}. Assume that for all $x \in Q$, we have $\|\nabla f(x, \xi)\|_* \leq M_f$, and $\|\nabla g_i(x, \zeta)\|_* \leq M_g \, \forall i =1, \ldots, m$, where $M_g = \max_{1 \leq i \leq m} \{M_{g_i}\}$. Then, for Algorithm \ref{alg:stocadaptive_mirror1_mod}, $I \ne \emptyset$, it stops after not more than
\[
    k = \left\lceil \frac{4 \max\{M_f^2, M_g^2\} \Theta_0^2 \left(1 + \frac{1}{\gamma}\right)^2}{\varepsilon^2} \right\rceil 
\]
iterations and returns \(\hat{x}\), such that
\begin{equation}\label{eq_results_theo2_mod}
    \mathbb{E}[f(\hat{x})] - f(x_*) \leq  \varepsilon, \quad \text{and} \quad g(\hat{x}) \leq \varepsilon.
\end{equation}
\end{theorem}

\begin{proof}
For each iteration \(i = 0, 1, \ldots, k-1\), define
\begin{equation}\label{nabla_i_Delta_i_stoc_mod}
    \nabla_i := 
    \begin{cases}
    \nabla f(x_i, \xi_i) & \text{if } \, i \in I,\\
    \nabla g_{q(i)}(x_i, \zeta_i) & \text{if } \, i \in J.
    \end{cases}  \Delta_i := \begin{cases}
    \nabla f(x_i, \xi_i) -\nabla f(x_i)  & \text{if } \, i \in I,\\
    \nabla g_{q(i)}(x_i, \zeta_i) - \nabla g_{q(i)}(x_i) & \text{if } \, i \in J.
    \end{cases} 
\end{equation}

At each iteration, we have
\begin{equation}\label{M_i_sto_mod}
    h_k = \Theta_0 \left(\sum_{i = 0}^{k} M_i^2\right)^{-1/2}, \quad M_i = \begin{cases}
    \|\nabla f(x_i, \xi_i)\|_* & \text{if } \, i \in I,\\
    \|\nabla g_{q(i)}(x_i, \zeta_i)\|_* & \text{if } \, i \in J.
    \end{cases}
\end{equation}
Let \(V_i := V(x_*, x_i)\). By Lemma \ref{main_lemma_stoch} (inequality \eqref{main_lemma_stoc}) with \(y = x_*\) at iteration \(i= 0, 1, \ldots, k-1\), and using \eqref{nabla_i_Delta_i_stoc_mod}, \eqref{M_i_sto_mod},  we get
\begin{equation}\label{eq_main_all_iters_stoc_mod}
     h_i \langle \nabla_i, x_i - x_* \rangle \leq \frac{h_i^2}{2} M_i^2 + V_i - V_{i+1}.
\end{equation}

For all $i \in I$, since $f$ is $\gamma$-quasar convex, from \eqref{eq_1lemma_stoch} with $\Delta_i = \nabla f(x_i, \xi_i) - \nabla f(x_i)$, we have
\begin{equation*}
    \langle \nabla f(x_i, \xi_i), x_i - x_* \rangle + \langle \nabla f(x_i, \xi_i) - \nabla f(x_i),  x_* - x_i \rangle \geq \gamma \left(f(x_i) - f(x_*)\right).
\end{equation*}

For all \(i \in J\), by the convexity of \(g_{q(i)}\) we have
\begin{equation*}
    \langle \nabla g_{q(i)}(x_i), x_i - x_* \rangle \ge g_{q(i)}(x_i) - g_{q(i)}(x_*).
\end{equation*}
Thus, 
\begin{align*}
    g_{q(i)}(x_i) - g_{q(i)}(x_*) & \leq \langle \nabla g_{q(i)}(x_i, \zeta_i), x_i - x_* \rangle  \nonumber
    \\& \quad + \langle \nabla g_{q(i)}(x_i, \zeta_i) - \nabla g_{q(i)}(x_i),  x_* - x_i \rangle.
\end{align*}

Thus, for all \(i = 0, 1, \ldots, k-1\), we have
\begin{equation}\label{hgh12_s_mod}
    \langle \nabla_i, x_i - x_* \rangle + \beta_i \geq \delta_i ,
\end{equation}
where
\[
    \delta_i := \begin{cases}
    \gamma \bigl(f(x_i) - f(x_*)\bigr) & \text{if } \, i \in I,
    \\g_{q(i)}(x_i) - g_{q(i)}(x_*) & \text{if } \, i \in J,
    \end{cases}
\]
and 
\[
    \beta_i :=  \begin{cases}
     \langle \nabla f(x_i, \xi_i) - \nabla f(x_i),  x_* - x_i \rangle & \text{if } \, i \in I,
     \\ \langle \nabla g_{q(i)}(x_i, \zeta_i) - \nabla g_{q(i)}(x_i),   x_* - x_i \rangle & \text{if } \, i \in J.
    \end{cases}
\]

Thus, from \eqref{eq_main_all_iters_stoc_mod} and \eqref{hgh12_s_mod}, for all \(i = 0, 1, \ldots, k-1\), we obtain
\begin{equation}\label{jhjkh_s_mod}
    \delta_i \le \frac{h_i}{2} M_i^2 + \frac{1}{h_i} \left(V_i - V_{i+1}\right) + \beta_i \quad \forall \, i = 0, 1, \ldots, k-1.
\end{equation}

Summing up \eqref{jhjkh_s_mod} for \(i = 0, 1, \dots, k-1\), we get
\begin{equation}\label{lkl1_s_mod}
    \sum_{i=0}^{k-1} \delta_i \le \sum_{i=0}^{k-1} \frac{h_i}{2} M_i^2 + \sum_{i=0}^{k-1} \frac{1}{h_i} \left(V_i - V_{i+1}\right) + \sum_{i=0}^{k-1} \beta_i.
\end{equation}

For the first sum in the right-hand side of \eqref{lkl1_s_mod}, we have (see \eqref{bound1})
\begin{equation}\label{bound1_s_mod}
    \sum_{i=0}^{k-1} \frac{h_i}{2} M_i^2 \leq \Theta_0 \left( \sum_{i=0}^{k-1} M_i^2 \right)^{1/2}.
\end{equation}

Also, for the second sum in the right-hand side of \eqref{lkl1_s_mod}, we have (see \eqref{bound2})
\begin{equation}\label{bound2_s_mod}
    \sum_{i=0}^{k-1} \frac{1}{h_i} (V_i - V_{i+1}) \leq \Theta_0 \left( \sum_{i=0}^{k-1} M_i^2 \right)^{1/2} + \Theta_0 M,
\end{equation}
where \(M := \max\{M_f, M_g\}\).

From \eqref{lkl1_s_mod}, \eqref{bound1_s_mod}, \eqref{bound2_s_mod}, and from stopping criterion, we get the following inequality
\begin{equation}\label{0123}
    \sum_{i=0}^{k-1} \delta_i \leq \frac{\varepsilon k}{1 + 1/\gamma} + \Theta_0 M + \sum_{i=0}^{k-1} \beta_i.
\end{equation}

Recall that for \(i \in J\), we have \(g_{q(i)}(x_i) > \varepsilon\) and \(g_{q(i)}(x_*) \le 0\), so
\[
    \delta_i = g_{q(i)}(x_i) - g_{q(i)}(x_*) \ge g_{q(i)}(x_i) > \varepsilon.
\]

We have \(|I| + |J| = k\). So, we get the following
\begin{align}\label{545ss_mod_mod}
    \sum_{i=0}^{k-1} \delta_i & = \sum_{i \in I} \gamma \bigl(f(x_i)-f(x_*)\bigr) + \sum_{i \in J} (g_{q(i)}(x_i)-g_{q(i)}(x_*)) \nonumber
    \\& \ge \gamma \sum_{i \in I} (f(x_i)-f(x_*)) + |J| \varepsilon. 
\end{align}

Hence, from \eqref{0123} and \eqref{545ss_mod_mod}, we find 
\[
    \gamma \sum_{i \in I} (f(x_i)-f(x_*)) \leq \frac{\varepsilon (\gamma |I| - |J|)}{ 1 + \gamma } + \Theta_0 M + \sum_{i=0}^{k-1} \beta_i.
\]

Thus, 
\[
    \sum_{i \in I} (f(x_i)-f(x_*)) \le \frac{\varepsilon |I|}{1 + \gamma } + \frac{\Theta_0 M}{\gamma} + \frac{1}{\gamma} \sum_{i=0}^{k-1} \beta_i.
\]

Consequently, we get 
\[
    f(\hat{x}) - f(x_*) \leq \frac{\varepsilon}{\gamma + 1} + \frac{\Theta_0 M}{\gamma |I|} + \frac{1}{\gamma} \sum_{i=0}^{k-1} \frac{\beta_i}{|I|}.
\]

Now, by taking the expectation, and since $\sum_{i=0}^{k-1} \mathbb{E} \left[\frac{\beta_i}{|I|}\right] = 0 $ (see \cite{Bayandina2017Stochastic}), we obtain
\[
    \mathbb{E}[f(\hat{x})] - f(x_*) \leq \frac{\varepsilon}{\gamma + 1} + \frac{\Theta_0 M}{\gamma |I|}.
\]

But, for sufficiently small $\varepsilon > 0$, in a similar way that was in the proof of Theorem \ref{theorem_stoc_adaptive}, we can prove the desired result,
$$
    \mathbb{E}[f(\hat{x})] - f(x_*) \leq  \varepsilon.
$$
In addition $g(\hat{x}) \leq \varepsilon$, and $I \ne \emptyset.$
\end{proof}

\section{Mirror Descent Algorithms for Quasar Convex Problems With Functional Constraints (Group II)}\label{section_group_II}

In this section, we propose the second group of algorithms for solving problems \eqref{general_problem_1} and \eqref{general_problem} in both deterministic and stochastic settings. Although structurally similar to the algorithms presented in Section \ref{section_group_I}, the proposed algorithms feature distinct step sizes and stopping criteria.

\subsection{Standard Adaptive Deterministic Mirror Descent Algorithm for Problems with One Functional Constraint}

\begin{alg}[Standard Adaptive Deterministic Mirror Descent Algorithm for Problems with One Functional Constraint]\label{alg_determ_II}
\hfill \break

\textbf{If} $g(x_k) \leq \varepsilon$,
$$
    M_k = \|\nabla f(x_k)\|_{*}, \quad h_k = \frac{\gamma \varepsilon}{M_k^2}, \quad x_{k+1} = \operatorname{Mirr}_{x_k} \left( h_k \nabla f(x_k) \right).
$$

\textbf{If} $g(x_k) > \varepsilon$,
$$
    M_k = \|\nabla g(x_k)\|_{*}, \quad h_k = \frac{\gamma \varepsilon}{M_k^2}, \quad x_{k+1} = \operatorname{Mirr}_{x_k} \left( h_k \nabla g(x_k) \right).
$$

The \textbf{Stopping criterion} is
\begin{equation}\label{stop_alg1_II}
    \sum_{i = 0}^{k-1} \frac{1}{M_i^2} \geq \frac{2 \Theta_0^2}{\gamma^2 \varepsilon^2}\left(1+ \frac{1}{\gamma}\right)^2. 
\end{equation}
\end{alg}

For Algorithm \ref{alg_determ_II}, we have the following result. 

\begin{theorem}\label{theo_alg1_determ_II}
Let $f: Q \longrightarrow \mathbb{R}$ be a Lipschitz continuous \eqref{Lip_f}, $\gamma$-quasar convex function \eqref{def_quasar}, and $g: Q \longrightarrow \mathbb{R}$ be a convex Lipschitz continuous function \eqref{Lip_g}. Then, with assumption \eqref{bound_bregman}, for Algorithm \ref{alg:adaptive_mirror1}, $I \ne \emptyset$, it stops after not more than
\begin{equation}\label{itersII}
    k = \left\lceil \frac{2 \Theta_0^2 \max\{M_f^2, M_g^2\}}{\gamma^2 \varepsilon^2}\left(1+ \frac{1}{\gamma}\right)^2 \right\rceil 
\end{equation}
iterations and returns \(\hat{x}= \arg\min_{ i \in I} f(x_i)\), such that
\begin{equation*}
    f(\hat{x}) - f(x_*) \leq  \varepsilon, \quad \text{and} \quad  g(\hat{x}) \leq \varepsilon.
\end{equation*}
\end{theorem}

\begin{proof}
At each iteration, we have
\begin{equation}\label{M_i_alg_determ_II}
     M_i = \begin{cases}
    \|\nabla f(x_i)\|_* & \text{if } \, i \in I,\\
    \|\nabla g(x_i)\|_* & \text{if } \, i \in J.
    \end{cases}
\end{equation}
Let \(V_i := V(x_*, x_i)\). Note that \(V_i \ge 0\) by \eqref{eq_breg}, and  \(V_i \le \Theta_0^2\) by \eqref{bound_bregman}.

For any $i \in I$, applying Lemma \ref{main_lemma_determ} (inequality \eqref{eq_2lemma}), and using \eqref{M_i_alg_determ_II}, we get
\begin{equation}\label{eq_main_iI_II}
     h_i (f(x_i) - f(x_*)) \leq \frac{1}{\gamma}\left(\frac{h_i^2}{2} M_i^2 + V_i - V_{i+1}\right).
\end{equation}
With $h_i = \frac{\gamma \varepsilon}{M_i^2} \; (\forall i \in I)$, we find 
\begin{equation}\label{4555200}
    \frac{\gamma \varepsilon}{M_i^2} (f(x_i) - f(x_*)) \leq \frac{\gamma \varepsilon^2}{2 M_i^2} + \frac{1}{\gamma} \left(V_i - V_{i+1}\right).
\end{equation}

Summing up both sides of \eqref{4555200}, over all $i \in I$, we find 
\begin{equation}\label{main_eqI}
    \gamma^2 \varepsilon  \sum_{i \in I} \frac{f(x_i) - f(x_*)}{M_i^2} \leq \frac{\gamma^2 \varepsilon^2}{2} \sum_{i \in I}  \frac{1}{M_i^2} + \sum_{i \in I}  \left(V_i - V_{i+1}\right). 
\end{equation}

For any $i \in J$, applying Lemma \ref{main_lemma_determ} (inequality \eqref{main_lemma_1}), we get
$$
    h_i \langle \nabla g(x_i), x_i - x_* \rangle \le \frac{h_i^2}{2} \|\nabla g(x_i)\|_*^2 + V_i - V_{i+1}. 
$$

Since $g$ is convex, $g(x_i) > \varepsilon \; (\forall i \in J)$, and $g(x_*) \leq 0$, we have 
$$
    \varepsilon < g(x_i) \leq g(x_i) - g(x_*) \leq \langle \nabla g(x_i), x_i - x_* \rangle,
$$
Thus, 
$$
     h_i \varepsilon \le \frac{h_i^2}{2} \|\nabla g(x_i)\|_*^2 + V_i - V_{i+1}. 
$$

With $h_i = \frac{\gamma \varepsilon}{M_i^2} \; (\forall i \in J)$, we find
$$
    \frac{\gamma \varepsilon^2}{M_i^2} < \frac{\gamma^2 \varepsilon^2}{2 M_i^2} + V_i - V_{i+1}.
$$
That is, 
$$
    \left(\gamma - \frac{\gamma^2}{2}\right) \frac{\varepsilon^2}{M_i^2} < V_i - V_{i+1} \quad \forall i \in J.
$$

Since $\gamma \le 1$, we find $\gamma - \frac{\gamma^2}{2} \ge \frac{\gamma}{2}$. Therefore, 
\begin{equation}\label{45552000}
    \frac{\gamma}{2} \frac{\varepsilon^2}{M_i^2} < V_i - V_{i+1} \quad \forall i \in J.
\end{equation}

Summing up both sides of \eqref{45552000}, over all $i \in J$, we find 
\begin{equation}\label{main_eqII}
   \frac{\gamma \varepsilon^2}{2}  \sum_{i \in J} \frac{1}{M_i^2} < \sum_{i \in J}  (V_i - V_{i+1}). 
\end{equation}

Let us now prove that $I \ne \emptyset$. For this, let us assume that $I = \emptyset$, then all iterations are non-productive, so from \eqref{main_eqII}, we can write
$$
    \frac{\gamma \varepsilon^2}{2}  \sum_{i = 0}^{k - 1} \frac{1}{M_i^2} < \sum_{i = 0}^{k - 1}  (V_i - V_{i+1}) \leq \Theta_0^2. 
$$
That is 
\begin{equation}\label{101cxz}
    \sum_{i = 0}^{k - 1} \frac{1}{M_i^2} < \frac{2 \Theta_0^2}{\gamma \varepsilon^2}. 
\end{equation}
But, form stopping criterion \eqref{stop_alg1_II}, and since $\frac{2}{\gamma^2} \left(1+ \frac{1}{\gamma}\right)^2 > \frac{2}{\gamma} \, \forall \gamma \in (0, 1]$, we find
$$
    \sum_{i = 0}^{k - 1} \frac{1}{M_i^2} \ge \frac{2 \Theta_0^2}{\gamma^2 \varepsilon^2} \left(1+ \frac{1}{\gamma}\right)^2 > \frac{2\Theta_0^2}{\gamma \varepsilon^2}. 
$$
This, we obtain a contradiction with \eqref{101cxz}. Hence $I \ne \emptyset$. 

Now, from \eqref{main_eqI}, and \eqref{main_eqII}, we find
\begin{align*}
    \gamma^2 \varepsilon  \sum_{i \in I} \frac{f(x_i) - f(x_*)}{M_i^2} + \frac{\gamma \varepsilon^2}{2}  \sum_{i \in J} \frac{1}{M_i^2} & \leq \frac{\gamma^2 \varepsilon^2}{2} \sum_{i \in I}  \frac{1}{M_i^2} + \sum_{i=0}^{k-1}  \left(V_i - V_{i+1}\right)
    \\& \leq \frac{\gamma^2 \varepsilon^2}{2} \sum_{i \in I} \frac{1}{M_i^2} + \Theta_0^2.
\end{align*}

Thus, 
\begin{equation*}
    \sum_{i \in I} \frac{f(x_i) - f(x_*)}{M_i^2}  \leq \frac{\varepsilon}{2} \sum_{i \in I} \frac{1}{M_i^2} + \frac{\Theta_0^2}{\gamma^2 \varepsilon} - \frac{\varepsilon}{2 \gamma} \sum_{i \in J} \frac{1}{M_i^2}.
\end{equation*}

We have $\hat{x} : = \arg\min_{i \in I} f(x_i)$. Thus, 
$$
    (f(\hat{x}) - f(x_*) ) \sum_{i \in I} \frac{1}{M_i^2} \leq \sum_{i \in I} \frac{f(x_i) - f(x_*)}{M_i^2} .
$$

Therefore, 
\begin{equation}\label{mainIII}
    (f(\hat{x}) - f(x_*) ) \sum_{i \in I} \frac{1}{M_i^2} \leq \frac{\varepsilon}{2} \sum_{i \in I} \frac{1}{M_i^2} + \frac{\Theta_0^2}{\gamma^2 \varepsilon} - \frac{\varepsilon}{2 \gamma} \sum_{i \in J} \frac{1}{M_i^2}. 
\end{equation}

From stopping criterion \eqref{stop_alg1_II}, we have 
$$
    \sum_{i \in I} \frac{1}{M_i^2} + \sum_{i \in J} \frac{1}{M_i^2} \ge \frac{2 \Theta_0^2}{\gamma^2 \varepsilon^2} \left(1+ \frac{1}{\gamma}\right)^2. 
$$
Thus, since $\left(1+ \frac{1}{\gamma}\right)^2  \ge \frac{1}{\gamma}$, we get 
\begin{equation}\label{main_VI}
    \frac{\Theta_0^2}{\gamma^2 \varepsilon}  \leq \left(\sum_{i \in I} \frac{1}{M_i^2} + \sum_{i \in J} \frac{1}{M_i^2}\right) \frac{\varepsilon \gamma}{2} \leq  \frac{\varepsilon \gamma }{2} \sum_{i \in I} \frac{1}{M_i^2} +\frac{\varepsilon}{2 \gamma }\sum_{i \in J} \frac{1}{M_i^2}. 
\end{equation}
where in the last, we used the fact $\frac{\gamma}{2} \le \frac{1}{2 \gamma},  \, \forall 0 < \gamma \le 1$.

Hence, from  \eqref{mainIII} and \eqref{main_VI}, we get 
\begin{equation}\label{mainV}
    (f(\hat{x}) - f(x_*) ) \sum_{i \in I} \frac{1}{M_i^2} \leq \frac{\varepsilon}{2} \sum_{i \in I} \frac{1}{M_i^2} + \frac{\varepsilon \gamma }{2} \sum_{i \in I} \frac{1}{M_i^2}. 
\end{equation}

Dividing by $\sum_{i \in I} \frac{1}{M_i^2} > 0$, since $I \ne \emptyset$, we get 
$$
    f(\hat{x}) - f(x_*) \leq \frac{\varepsilon}{2} + \frac{\varepsilon \gamma }{2} \leq \frac{\varepsilon}{2} + \frac{\varepsilon}{2} = \varepsilon. 
$$

From the definition of $\hat{x} = \arg\min_{i \in I} f(x_i)$, it s clear that $g(\hat{x}) \leq \varepsilon$. 

At the end, from \eqref{Lip_f} and \eqref{Lip_g}, we have $M_i^2 \leq \max \{M_f^2, M_g^2\}$ for all $i = 0, 1 , \ldots , k -1$. Thus, by \eqref{itersII}, we find 
$$
    \sum_{i = 0}^{k-1} \frac{1}{M_i^2} \geq  \frac{k}{\max\{M_f^2, M_g^2\}} \ge \frac{2 \Theta_0^2}{\gamma^2 \varepsilon^2}\left(1+ \frac{1}{\gamma}\right)^2.
$$
Hence, the inequality in the stopping criterion \eqref{stop_alg1_II} holds for $k$ defined in \eqref{itersII}.  This completes the proof. 
\end{proof}

\subsection{Modification of Standard Adaptive Deterministic Mirror Descent Algorithm for Problems with Many Functional Constraints.}\label{subsec_modif_groupII}

To solve problem \eqref{general_problem_1} with a large number of functional constraints, we modify Algorithm \ref{alg_determ_II} in a manner similar to that in Subsection \ref{subsec_modif_groupI}. The proposed modification is presented as Algorithm \ref{alg_determ_II_mod}.

\begin{alg}[Modification of Adaptive Deterministic Mirror Descent Algorithm for Problems With Many Functional Constraint]\label{alg_determ_II_mod}
\hfill \break

\textbf{If} $g_i(x_k) \leq \varepsilon \, \forall i = 1, \ldots, m$,
$$
    M_k = \|\nabla f(x_k)\|_{*}, \quad h_k = \frac{\gamma \varepsilon}{M_k^2}, \quad x_{k+1} = \operatorname{Mirr}_{x_k} \left( h_k \nabla f(x_k) \right).
$$

\textbf{If} there exist $q = q(k) \in \{1, \ldots, m\}$ such that $g_{q(k)}(x_k) > \varepsilon$,
$$
    M_k = \|\nabla g_{q(k)}(x_k)\|_{*}, \quad h_k = \frac{\gamma \varepsilon}{M_k^2}, \quad x_{k+1} = \operatorname{Mirr}_{x_k} \left( h_k \nabla g_{q(k)}(x_k) \right).
$$

The \textbf{Stopping criterion} is
\begin{equation}\label{stop_alg1_II_mod}
    \sum_{i = 0}^{k-1} \frac{1}{M_i^2} \geq \frac{2 \Theta_0^2}{\gamma^2 \varepsilon^2}\left(1+ \frac{1}{\gamma}\right)^2. 
\end{equation}
\end{alg}

For Algorithm \ref{alg_determ_II_mod}, we have the following result. 

\begin{theorem}\label{theo_alg1_determ_II_mod}
Let $f: Q \longrightarrow \mathbb{R}$ be a Lipschitz continuous \eqref{Lip_f}, $\gamma$-quasar convex function \eqref{def_quasar}, and $g_i: Q \longrightarrow \mathbb{R}$ be convex Lipschitz continuous functions with constants $M_{g_i}$ \eqref{Lip_g}. Assume that for all $x \in Q$, $\|\nabla f(x) \|_* \leq M_f$, and $\|\nabla g_i(x)\|_* \leq M_g \, \forall i = 1, \ldots, m$ where $M_g = \max_{1 \leq i \leq m} \{ M_{g_i} \}$. Then, with assumption \eqref{bound_bregman}, for Algorithm \ref{alg_determ_II_mod}, $I \ne \emptyset$, it stops after not more than
\begin{equation*}
    k = \left\lceil \frac{2 \Theta_0^2 \max\{M_f^2, M_g^2\}}{\gamma^2 \varepsilon^2}\left(1+ \frac{1}{\gamma}\right)^2 \right\rceil 
\end{equation*}
iterations and returns \(\hat{x}= \arg\min_{ i \in I} f(x_i)\), such that
\begin{equation*}
    f(\hat{x}) - f(x_*) \leq  \varepsilon, \quad \text{and} \quad  g(\hat{x}) \leq \varepsilon.
\end{equation*}
\end{theorem}
\begin{proof}
The proof follows the same structure as that of Theorem \ref{theo_alg1_determ_II}, with a minor modification in the part concerning non-productive steps. Specifically, in analogy with the proof of Theorem \ref{theo_modifalg1_determ}, instead of using $g(\cdot) = \max_{1 \leq i \leq m} \{g_i(\cdot)\}$, we consider the constraint $g_{q(k)}$ that violates feasibility.
\end{proof}

\subsection{Standard Non-Adaptive Deterministic Mirror Descent Algorithm for Problems with One Functional Constraint}

\begin{alg}[Standard Non-Adaptive Deterministic Mirror Descent Algorithm for Problems with One Functional Constraint]\label{alg2_determ_II}
\hfill \break

\textbf{Step size.} Let $M: = \max\{M_f, M_g\}$, then $h_k := h = \frac{\gamma \varepsilon }{M^2}$

\smallskip 

\textbf{If} $g(x_k) \leq \varepsilon$,
$$
    M_k = \|\nabla f(x_k)\|_{*}, \quad x_{k+1} = \operatorname{Mirr}_{x_k} \left( h \nabla f(x_k) \right).
$$

\textbf{If} $g(x_k) > \varepsilon$,
$$
    M_k = \|\nabla g(x_k)\|_{*},  \quad x_{k+1} = \operatorname{Mirr}_{x_k} \left( h \nabla g(x_k) \right).
$$

The \textbf{Stopping criterion} is
\begin{equation}\label{stop_alg2_II}
    k \geq \frac{2 \Theta_0^2 M^2}{\gamma^2 \varepsilon^2}\left(1+ \frac{1}{\gamma}\right)^2. 
\end{equation}
\end{alg}

Here, we note that for solving problem \eqref{general_problem_1} with a large number of functional constraints, Algorithm \ref{alg2_determ_II} can be modified in a manner analogous to those presented in Subsections \ref{subsec_modif_groupI} and \ref{subsec_modif_groupII}. For the resulting modified algorithm, the same result as in Theorem \ref{theo_alg2_determ_II} holds, with an identical proof except for a minor adjustment in the treatment of non-productive steps. Specifically, in analogy with the previous modifications, we replace $g(\cdot) = \max_{1 \leq i \leq m} \{g_i(\cdot)\}$ with the particular violating constraint $g_{q(k)}$.

\medskip 

For Algorithm \ref{alg2_determ_II}, we have the following result. 

\begin{theorem}\label{theo_alg2_determ_II}
Let $f: Q \longrightarrow \mathbb{R}$ be a Lipschitz continuous \eqref{Lip_f}, $\gamma$-quasar convex function \eqref{def_quasar}, and $g: Q \longrightarrow \mathbb{R}$ be a convex Lipschitz continuous function \eqref{Lip_g}. Then, with assumption \eqref{bound_bregman}, for Algorithm \ref{alg2_determ_II}, $I \ne \emptyset$, it stops after not more than
\begin{equation}\label{itersII_stoc}
    k = \left\lceil \frac{2 \Theta_0^2 M^2}{\gamma^2 \varepsilon^2}\left(1+ \frac{1}{\gamma}\right)^2 \right\rceil 
\end{equation}
iterations and returns \(\hat{x}= \arg\min_{ i \in I} f(x_i)\), such that
\begin{equation*}
    f(\hat{x}) - f(x_*) \leq  \varepsilon, \quad \text{and} \quad  g(\hat{x}) \leq \varepsilon.
\end{equation*}
\end{theorem}
\begin{proof}
The proof is the same of the proof of Theorem \ref{theo_alg1_determ_II}. See also the proof of the next theorem for the stochastic version of Algorithm \ref{alg2_determ_II}. 
\end{proof}

\subsection{Standard Non-Adaptive Stochastic Mirror Descent Algorithm for Problems with One Functional Constraint.}\label{subsec_stoc_stand_GII}

To solve the stochastic formulation of problem \eqref{general_problem}, we propose, in a manner similar to Subsection \ref{subsec_stoc_stand_GI}, a stochastic version of Algorithm \ref{alg2_determ_II}, presented as Algorithm \ref{alg2_stoc_II}. Unfortunately, it is not evident how to prove convergence for a stochastic version of the adaptive Algorithm \ref{alg_determ_II}. This limitation is the primary reason for introducing the non-adaptive Algorithm \ref{alg2_stoc_II}, as its stochastic counterpart admits a rigorous convergence analysis.

\begin{alg}[Standard Non-Adaptive Stochastic Mirror Descent Algorithm for Problems with One Functional Constraint]\label{alg2_stoc_II}
\hfill \break

\textbf{Step size.} Let $M: = \max\{M_f, M_g\}$, then $h_k := h = \frac{\gamma \varepsilon }{M^2}$

\smallskip 

\textbf{If} $g(x_k) \leq \varepsilon$,
$$
    M_k = \|\nabla f(x_k, \xi_k)\|_{*},  \quad x_{k+1} = \operatorname{Mirr}_{x_k} \left( h_k \nabla f(x_k, \xi_k) \right).
$$

\textbf{If} $g(x_k) > \varepsilon$,
$$
    M_k = \|\nabla g(x_k, \zeta_k)\|_{*}, \quad x_{k+1} = \operatorname{Mirr}_{x_k} \left( h_k \nabla g(x_k, \zeta_k) \right).
$$

The \textbf{Stopping criterion} is
\begin{equation}\label{stop_alg1_II_stoc}
    k \geq \frac{2 \Theta_0^2M^2}{\gamma^2 \varepsilon^2}\left(1+ \frac{1}{\gamma}\right)^2. 
\end{equation}
\end{alg}

Here also, we note that for problem \eqref{general_problem_1} in the stochastic setting involving a large number of functional constraints, Algorithm \ref{alg2_stoc_II} admits a modification similar to that in Subsection \ref{modif_stoc_GI}. The modified algorithm achieves the same convergence guarantee as Theorem~\ref{theo_alg1_stoc_II}, with a proof that is otherwise identical except for the treatment of non-productive steps, where we replace $g(\cdot) = \max_{1 \leq i \leq m} \{g_i(\cdot)\}$ with the violating constraint $g_{q(k)}$.

\medskip 

For Algorithm \ref{alg2_stoc_II}, we have the following result. 

\begin{theorem}\label{theo_alg1_stoc_II}
Let $f: Q \longrightarrow \mathbb{R}$ be a $\gamma$-quasar convex function \eqref{def_quasar}, and $g: Q \longrightarrow \mathbb{R}$ be a convex Lipschitz continuous function. Assume that \eqref{bound_stoc_f_g_1}, and \eqref{bound_stoc_f_g_2} hold. Then, with assumption \eqref{bound_bregman}, for Algorithm \ref{alg2_stoc_II}, $I \ne \emptyset$, it stops after not more than
\begin{equation*}\label{itersII_stoc_s}
    k = \left\lceil \frac{2 \Theta_0^2 M^2}{\gamma^2 \varepsilon^2}\left(1+ \frac{1}{\gamma}\right)^2 \right\rceil 
\end{equation*}
iterations and returns \(\hat{x}= \arg\min_{ i \in I} f(x_i)\), such that
\begin{equation*}
    \mathbb{E}[f(\hat{x})] - f(x_*) \leq  \varepsilon, \quad \text{and} \quad  g(\hat{x}) \leq \varepsilon.
\end{equation*}
\end{theorem}

\begin{proof}
Let us define 
\[
    \beta_i := 
    \begin{cases}
     \langle \nabla f(x_i, \xi_i) - \nabla f(x_i),  x_* - x_i \rangle & \text{if } \, i \in I,\\
     \langle \nabla g(x_i, \zeta_i) - \nabla g(x_i),   x_* - x_i \rangle & \text{if } \, i \in J.
    \end{cases}
\]
Also, let \(V_i := V(x_*, x_i)\). 

For any $i \in I$, applying Lemma \ref{main_lemma_stoch} (inequality \eqref{eq_2lemma_stoch} with $\Delta_i = \nabla f(x_i, \xi_i) - \nabla f(x_i)$), we get
\begin{align*}
     & \quad \; h (f(x_i) - f(x_*)) + \frac{h}{\gamma} \langle \nabla f(x_i, \xi_i) - \nabla f(x_i), x_i - x_*  \rangle  
     \\& \leq \frac{1}{\gamma}\left(\frac{h^2}{2} \|\nabla f(x_i, \xi_i)\|_*^2 + V_i - V_{i+1}\right).
\end{align*}
Since $\left\langle \nabla f(x_i,\xi_i)-\nabla f(x_i), x_i-x_* \right\rangle =-\beta_i,$ and $\|\nabla f(x_i,\xi_i)\|_*\le M,$ we get
\begin{equation}\label{zz1}
    h(f(x_i)-f(x_*)) \le \frac1\gamma \left( \frac{h^2M^2}{2} + V_i-V_{i+1} \right) + \frac{h}{\gamma}\beta_i.
\end{equation}

Summing up both sides of \eqref{zz1}, over all $i\in I$, we obtain
\begin{align} \label{eq_prod}
    h\sum_{i\in I}(f(x_i)-f(x_*)) \le \frac{|I|h^2M^2}{2\gamma} + \frac1\gamma \sum_{i\in I}(V_i-V_{i+1}) + \frac{h}{\gamma} \sum_{i\in I}\beta_i.
\end{align}

For any $i \in J$, applying Lemma \ref{main_lemma_stoch} (inequality \eqref{main_lemma_stoc} with $\Delta_i = \nabla g(x_i, \zeta_i) - \nabla g(x_i)$, and $y = x_*$), we get
$$
    h \langle \nabla g(x_i, \zeta_i) , x_i - x_* \rangle \le \frac{h^2}{2} \|\nabla g(x_i, \zeta_i)\|_*^2 + V_i - V_{i+1}. 
$$
Thus, 
$$
    h \langle \nabla g(x_i), x_i - x_* \rangle + h \langle \nabla g(x_i, \zeta_i) - \nabla g(x_i), x_i - x_* \rangle \le \frac{h^2}{2} \|\nabla g(x_i, \zeta_i)\|_*^2 + V_i - V_{i+1}. 
$$

Since $ \langle \nabla g(x_i, \zeta_i) - \nabla g(x_i), x_i - x_* \rangle = -\beta_i$,  and $\|\nabla g(x_i,\zeta_i)\|_*\le M,$ we get
\[
    h \left\langle \nabla g(x_i), x_i-x_* \right\rangle \le \frac{h^2M^2}{2} + V_i-V_{i+1} + h\beta_i.
\]

Since $g$ is convex, $g(x_i)>\varepsilon$ for all $i\in J$, and $g(x_*)\le0$, we get
\[
    \varepsilon < g(x_i)-g(x_*) \le \left\langle \nabla g(x_i), x_i-x_* \right\rangle.
\]
Therefore,
\begin{equation}\label{wwx1}
    h\varepsilon  <  \frac{h^2M^2}{2} + V_i-V_{i+1} + h\beta_i.
\end{equation}

Summing up both sides of \eqref{wwx1}, over all $i\in J$, we obtain
\begin{align}\label{eq_nonprod}
    |J|h\varepsilon < \frac{|J|h^2M^2}{2} + \sum_{i\in J}(V_i- V_{i+1}) + h\sum_{i\in J}\beta_i.
\end{align}

Adding \eqref{eq_prod} and \eqref{eq_nonprod}, we obtain
\begin{align*}
    h\sum_{i\in I}(f(x_i)-f(x_*)) + |J|h \varepsilon & \le
    \frac{|I|h^2M^2}{2\gamma} + \frac{|J|h^2M^2}{2} + \sum_{i=0}^{k-1}(V_i-V_{i+1}) 
    \\& \quad + \frac{h}{\gamma}\sum_{i\in I}\beta_i +h\sum_{i\in J}\beta_i.
\end{align*}

Since $0<\gamma\le1$, we have $\frac{|J|h^2M^2}{2} \le \frac{|J|h^2M^2}{2\gamma}$, and $h \le \frac{h}{\gamma}$. Hence,
\begin{align}\label{eq_main}
    h\sum_{i\in I}(f(x_i)-f(x_*)) + |J|h\varepsilon & \le \frac{kh^2M^2}{2\gamma} + \Theta_0^2 + \frac{h}{\gamma} \sum_{i=0}^{k-1}\beta_i \nonumber
    \\& = \frac{kh\varepsilon}{2} + \Theta_0^2 + \frac{h}{\gamma} \sum_{i=0}^{k-1}\beta_i.
\end{align}

From the definition of $\hat{x}:=\arg\min_{i\in I}f(x_i)$, and \eqref{eq_main}, we obtain 
\begin{equation}\label{eq_main_main}
    |I|(f(\hat{x})-f(x_*)) + |J|\varepsilon \le \frac{k\varepsilon}{2} + \frac{\Theta_0^2}{h} + \frac{1}{\gamma} \sum_{i=0}^{k-1}\beta_i. 
\end{equation}

By taking the expectation for both sides of \eqref{eq_main_main}, and since $\mathbb{E} \left[\sum_{i=0}^{k-1}\beta_i\right] = 0$, we get 
\begin{align*}
    |I|(\mathbb{E}[f(\hat{x})]-f(x_*)) + |J| \varepsilon \le \frac{k \varepsilon}{2} + \frac{\Theta_0^2}{h}.
\end{align*}

That is, 
\begin{align}\label{eq_hat_hat}
    |I|(\mathbb{E}[f(\hat{x})]-f(x_*)) \le \frac{|I|\varepsilon}{2} - \frac{|J|\varepsilon}{2} + \frac{\Theta_0^2}{h}.
\end{align}

From the stopping criterion $k \ge\frac{2\Theta_0^2M^2}{ \gamma^2\varepsilon^2} \left(1+\frac1\gamma\right)^2,$ and since $ \frac{2}{\gamma^2} \left(1+\frac1\gamma\right)^2 > \frac{2}{\gamma} \;\forall \gamma \in (0, 1]$, we obtain $\frac{\Theta_0^2}{h} \le \frac{k \varepsilon}{2} = \frac{(|I| + |J|) \varepsilon}{2} $. Hence, from \eqref{eq_hat_hat} we obtain the desired result, 
\[
    \mathbb{E}[f(\hat{x})]-f(x_*) \le \varepsilon.
\]

Since $\hat{x}$ is chosen among productive iterates, we obtain $g(\hat{x})\le \varepsilon.$

Finally, let us prove that $I\neq \emptyset$.  Assume by contradiction that $I=\emptyset$. Then all iterations are non-productive, and therefore $|J|=k$. 

From \eqref{eq_nonprod}, we obtain
\begin{equation*}
    kh\varepsilon  < \frac{k h^2 M^2}{2} + \sum_{i=0}^{k-1}(V_i- V_{i+1}) + h\sum_{i=0}^{k-1}\beta_i \leq \frac{kh^2M^2}{2} + \Theta_0^2 + h\sum_{i=0}^{k-1}\beta_i.
\end{equation*}

Taking the expectation of both sides of the previous inequality, and using $\mathbb{E}\left[\sum_{i=0}^{k-1}\beta_i\right]=0,$ we get
\[
    k h \varepsilon < \frac{k h^2 M^2}{2} + \Theta_0^2 \Longrightarrow \frac{k \gamma \varepsilon^2}{M^2} < \frac{k \gamma^2 \varepsilon^2}{2 M^2} + \Theta_0^2 \Longrightarrow \frac{k \gamma \varepsilon^2}{M^2} \left( 1-\frac{\gamma}{2} \right) < \Theta_0^2.
\]

Since $0<\gamma\le1$, we have $1-\frac{\gamma}{2}\ge \frac{1}{2}.$ Thus,
\begin{equation}\label{fdfd00}
    \frac{k\gamma\varepsilon^2}{2M^2} < \Theta_0^2 \quad \Longrightarrow \quad  k < \frac{2\Theta_0^2M^2}{\gamma\varepsilon^2}.
\end{equation}

On the other hand, From the stopping criterion $k \ge \frac{2 \Theta_0^2 M^2}{\gamma^2 \varepsilon^2} \left(1 + \frac{1}{\gamma}\right)^2,$ and Since $\frac1{\gamma^2}\left(1+\frac1\gamma\right)^2 > \frac1\gamma
\, \forall \gamma\in(0,1],$ we obtain $k> \frac{2 \Theta_0^2 M^2}{\gamma \varepsilon^2},$ which contradicts \eqref{fdfd00}. Hence, $I\neq \emptyset.$ This completes the proof.
\end{proof}


\section{Conclusions}

In this paper, we studied quasar convex optimization problems with functional (inequality-type) constraints. We proposed and analyzed the convergence rates of two groups of algorithms (standard and modified versions for each algorithm), each comprising different mirror descent-type methods in deterministic and stochastic settings with switching between productive and non-productive iteration points. As directions for future research, we plan to extend our approach to problems with strongly quasar convex objectives. Additionally, we aim to analyze the proposed algorithms under an inexact oracle framework, specifically using $\delta$-subgradients.


\end{document}